\documentclass{amsart}

\usepackage{times, amsmath,amssymb,amsthm}
\usepackage{graphicx}
\usepackage{a4wide}
\usepackage[all]{xy}

\def\bbb#1{\textcolor{blue}{#1}}

\usepackage{color}
\definecolor{Chocolat}{rgb}{0.36, 0.2, 0.09}
\definecolor{BleuTresFonce}{rgb}{0.215, 0.215, 0.36}

\usepackage[colorlinks,final,hyperindex]{hyperref}
\hypersetup{citecolor=BleuTresFonce, linkcolor=Chocolat}

\DeclareMathAlphabet{\mathbbold}{U}{bbold}{m}{n}

\DeclareSymbolFont{rsfscript}{OMS}{rsfs}{m}{n}
\DeclareSymbolFontAlphabet{\mathrsfs}{rsfscript}

\DeclareFontFamily{OMS}{rsfs}{\skewchar\font'177}
\DeclareFontShape{OMS}{rsfs}{m}{n}{%
      <5> rsfs5
      <6> <7> rsfs7
      <8> <9> <10> rsfs10
      <10.95> <12> <14.4> <17.28> <20.74> <24.88> rsfs10
      }{}
\def\a{\mathfrak{a}}
\newcommand{\ad}{\mathrm{ad}}
\newcommand{\MC}{\mathrm{MC}}
\newcommand{\Tw}{\mathrm{Tw}}
\newcommand{\epi}{\twoheadrightarrow}
\newcommand{\cH}{\mathcal{H}}
\newcommand{\mono}{\rightarrowtail}
\def\Im{\mathrm{Im}}
\def\Ker{\mathrm{Ker}}
\newcommand{\NN}{\mathbb{N}}
\newcommand{\Sy}{\mathbb{S}}
\def\KK{\mathbb{K}}
\newcommand{\ac}{\scriptstyle \text{\rm !`}}

\def\g{\mathfrak{g}}
\def\TTT{\mathcal{T}}
\def\la{\langle}
\def\ra{\rangle}
\def\qi{\xrightarrow{\sim}}
\newcommand{\Tree}{\mathsf{Tree}}
\def\C{\mathcal{C}}
\def\P{\mathcal{P}}
\def\I{\mathrm{I}}

\def\calM{\mathcal{M}}

\DeclareMathOperator{\id}{id}
\DeclareMathOperator{\End}{End}
\DeclareMathOperator{\Hom}{Hom}
\DeclareMathOperator{\sgn}{sgn}

\DeclareMathOperator{\BV}{BV}
\DeclareMathOperator{\sBV}{sBV}
\DeclareMathOperator{\trBV}{tBV}
\DeclareMathOperator{\Com}{Com}
\DeclareMathOperator{\HyperCom}{HyperCom}

\makeatletter

\theoremstyle{plain}

\newtheorem {theorem}{Theorem}
\newtheorem {lemma}{Lemma}

\newtheorem {proposition}{Proposition}
\newtheorem*{theoremintro}{Theorem}

\theoremstyle{definition}

\newtheorem {definition}{Definition}

\newtheorem*{remark}{\sc Remark}



\thanks{S.S. was supported by the Netherlands Organisation for Scientific Research. B.V. was supported by the ANR HOGT grant.}

\begin{document}

\title[Givental Action and Trivialisation of Circle Action]{Givental Action and Trivialisation of Circle Action} 

\author{Vladimir Dotsenko}
\address{School of Mathematics, Trinity College, Dublin 2, Ireland}
\email{vdots@maths.tcd.ie}

\author{Sergey Shadrin}
\address{Korteweg-de Vries Institute for Mathematics, University of Amsterdam, P. O. Box 94248, 1090 GE Amsterdam, The Netherlands}
\email{s.shadrin@uva.nl}

\author{Bruno Vallette}
\address{University Nice Sophia Antipolis, CNRS,  LJAD, UMR 7351, 06100 Nice, France.}
\email{brunov@unice.fr}

\begin{abstract}
In this paper, we show that the Givental group action on genus zero cohomological field theories, also known as formal Frobenius manifolds or hypercommutative algebras, 
naturally arises in the deformation theory of Batalin--Vilkovisky algebras. We prove that the Givental action 
is equal to an action of the trivialisations of the trivial circle action. 
This result relies on the equality of two Lie algebra actions coming from two apparently remote domains: geometry and homotopical algebra. 
\end{abstract}

\maketitle

\setcounter{tocdepth}{1}
\tableofcontents

\section*{Introduction}\label{sec:intro}

In this paper, we study in depth symmetries of algebras over the homology operad of the  moduli spaces of genus $0$ stable curves $H_\bullet(\overline{\mathcal{M}}_{0, n+1})$, known, in different contexts and with small differences in definitions,  under the names of a hypercommutative algebra, a formal Frobenius manifold, a genus $0$ reduction of Gromov--Witten theory, or a genus~$0$ cohomological field theory, see \cite{Manin99}. This structure plays a crucial role in a range of questions arising in string theory, enumerative algebraic geometry, and integrable hierarchies, and is one of the basic structures underlying the classical mirror phenomena in genus $0$. In this paper, we choose to take an \emph{algebraic} viewpoint on this structure, and refer to its instances as \emph{hypercommutative algebras}.

Algebraically, a hypercommutative algebra structure on a graded vector space $A$ is a representation of the operad $\HyperCom:=H_\bullet(\overline{\mathcal{M}}_{0, n+1})$ in the endomorphism operad of~$A$. The operad $\HyperCom$ is well studied; for instance, it is proved to be Koszul in~\cite{Getzler95}. For our purposes, it is important that its internal structure can be described in terms of the intersection theory on the moduli spaces of curves of genus~$0$ and that it is a crucial ingredient in the homotopy theory of Batalin--Vilkovisky algebras \cite{Manin99, DrummondColeVallette11, KhoroshkinMarkarianShadrin13}.

Let us consider the space of hypercommutative algebra structures on a given vector space $A$. Looking at the universal structure of the localisation formulas in Gromov--Witten theory, Givental observed in \cite{Givental01, Givental01bis} that this space is equipped with an action of a big group of ``formal Taylor loops of $GL(A)$'', which we call in this context the \emph{Givental group}. The Givental group is the main tool used to study various universal properties of hypercommutative algebras, and is behind important results in cohomological field theory and its relations to integrable hierarchies, matrix models, mirror symmetry, and homotopical algebra.

\smallskip

Representations of any operad admit a canonical deformation theory,  along the lines of \cite{MerkulovVallette09I}, but  the resulting group action is not an action of such a big group. This phenomenon certainly deserves a conceptual explanation. One such explanation was announced by M.~Kontsevich in 2003, see also~\cite{Costello05}. Basically, Kontsevich made two claims which together explain what is so special about the operad of hypercommutative algebras. First, he claimed that the operad of moduli spaces of genus~$0$ stable curves is a homotopy quotient of the operad of framed little disks by the circle action. Second, he announced that the arising action of changes of trivialisations of a homotopically trivial circle action  on representations of the operad $\HyperCom$ coincides with the action of the Givental group. The first statement is established in \cite{DrummondColeVallette11, DrummondCole11, KhoroshkinMarkarianShadrin13}, and the identification of the Givental action with the action of trivialisations of the circle action is obtained in \cite{KhoroshkinMarkarianShadrin13}. \\

In this paper, we show that the Givental action naturally arises in the deformation theory of  Batalin--Vilkovisky (BV) algebras. In particular, it allows us to prove the above-mentioned claim of Kontsevich about the Givental action in a functorial way, which does not require any computations on the level of the vector space $A$. Moreover, in the more general set-up of homotopy hypercommutative algebras, it turns out that this claim can be formulated in a simpler way, i.e. using trivialisations of a (strictly) trivial circle action.  
\smallskip 

To develop the deformation theory of BV-algebras, we require the more general notion of homotopy BV-algebras. 
Roughly speaking, the data of a homotopy BV-algebra is made up of a homotopy hypercommutative algebra and a homotopy circle action, see \cite{DrummondColeVallette11}. For any Koszul operad, there exists  a dg Lie algebra which controls the deformation theory 
of its algebras, see \cite[Chapter~$12$]{LodayVallette12}. It turns out that in the case of homotopy BV-algebras, one needs to use a certain homotopy Lie algebra $\mathfrak{l}_{\BV}$ . 

\bigskip

The first step we undertake amounts to re-interpreting the infinitesimal Givental action.
 A representation of the operad $\HyperCom$ in a given vector space $A$ is encoded by a Maurer--Cartan element in a certain dg Lie algebra~$\g_{\HyperCom}$. 
 It so happens that this dg Lie algebra is a subalgebra of the homotopy Lie algebra $\mathfrak{l}_{\BV}$, which is  
  an extension of $\g_{\HyperCom}$ by a dg Lie algebra~$\g_{\Delta}=z\End(A)[[z]]$. The degree zero elements of the latter Lie algebra form the Lie algebra of the Givental group. In a homotopy Lie algebra, the degree zero elements define vector fields on the variety of Maurer--Cartan elements, which are infinitesimal gauge symmetries of the Maurer--Cartan elements~\cite{Getzler09}. The first main result of this paper is the following theorem.

\begin{theoremintro}[Thm.~\ref{thm:MainGiv=LinftyAction}]
For any hypercommutative algebra structure on a graded vector space $A$ encoded by a Maurer--Cartan element $\alpha\in\g_{\HyperCom}$ and for any degree $0$ element $r(z)\in\g_{\Delta}$, the infinitesimal Givental action of $r(z)$ on $\alpha$ is equal to the gauge symmetry action of $r(z)$ on $\alpha$ viewed as a homotopy $\BV$-algebra structure:
$$\widehat{r(z)}.\alpha=\ell^\alpha_1(r(z))\ . $$
\end{theoremintro}

We proceed with integrating the infinitesimal action to a group action. 
This allows us to interpret the Givental group in a functorial way using operads as follows. 
We consider the operad $\trBV_\infty$ encoding homotopy BV-algebras together with a bit of extra data: a
trivialisation of the (homotopy) circle action.
We use gauge symmetries to construct  a quasi-isomorphism 
$$G\colon \HyperCom_\infty \to\trBV_\infty\ ,$$
which we call the \emph{Givental morphism}. (This gives  another proof of the first claim of Kontsevich). 
Using the Givental morphism, 
we define a morphism $\widetilde{G}$ from the operad $\HyperCom_\infty$ to  the quotient of the operad $\trBV_\infty$ that encodes homotopy hypercommutative algebras together with a trivialisation of the trivial circle action. 

\smallskip

Let us now explain how this approach allows us to prove the second claim of Kontsevich. 
For a graded vector space $A$, the Givental group can be identified with the group of trivialisations of the trivial circle action on $A$. 
Pulling back a hypercommutative algebra structure on $A$ by the morphism
$\widetilde{G}$, we recover the Givental group action. 

\begin{theoremintro}[Thm.~\ref{thm:MainII}]
Let $\alpha$ be a hypercommutative algebra and let $R(z)$ be a trivialisation of the trivial circle action. The pullback hypercommutative algebra ${\widetilde{G}}^*(\alpha, R(z)-1)$  is equal to the hypercommutative algebra, or CohFT, obtained by the Givental group action of the element $R(z)$. 
\end{theoremintro}

Note that the results of this paper provide us with a direct generalisation of the Givental action to an ($\infty$-groupoid) action on homotopy hypercommutative algebra structures.  Besides its structural importance for the foundations of the Givental theory, this  
creates a framework for  the development of the Gromov--Witten theory on the chain level: a theory where the space of Gromov--Witten classes may be a chain complex (think about the evaluations of differential forms rather than the cohomology classes) that produces a homotopy $\HyperCom$-algebra structure on the cohomology of the target variety.

\subsection*{Layout.} The paper is organised as follows. In Section~\ref{sec:Recol}, we recall the definitions, constructions, and results used in the sequel. Since our main results rely on bringing together methods from two fairly different areas, we made that section quite elaborate to benefit readers coming from either of those areas.  In Section~\ref{sec:HoLie-BV}, we define a particular $L_\infty$-algebra that controls the deformation theory of homotopy BV-algebras. In Section~\ref{sec:GaugeInterpretation}, we identify the infinitesimal Givental action with gauge symmetries inside that  $L_\infty$-algebra.  Finally, in Section~\ref{sec:Givental=Triv}, we integrate the infinitesimal action, give a functorial construction of the Givental group action, and prove that it amounts to the action of trivialisations of the trivial circle action. 
 
\subsection*{Acknowledgements.} This paper was completed during the authors\rq{} stays at University of Amsterdam, Newton Institute for Mathematical Sciences at the University of Cambridge (supported through the programme ``Grothendieck--Teichm\"uller Groups, Deformation and Operads''), Trinity College Dublin (supported through ``Visiting Professorships and Fellowships Benefaction Fund''), and University  Nice Sophia Antipolis. The authors would like to thank these institutions for the excellent working conditions enjoyed during their stay there. The second author would like to thank A.~Losev for the useful discussion of a possible set-up for the  Gromov--Witten theory on the chain level. 

\section{Recollections}\label{sec:Recol}

In this section, we recall necessary background information from various areas invoked in this paper. We assume working knowledge of standard results of homotopical algebra for operads, and encourage the reader to consult~\cite{LodayVallette12} for details on that.

Throughout the text, we work over a field $\KK$ of characteristic $0$. We denote by $s$ the suspension  operator of degree $1$: $(sC)_{\bullet+1}:=sC_\bullet$. We use the ``topologist's notation'' for finite sets, putting $\underline{n}:=\{1, \ldots , n\}$. The notation $\odot$ stands for the `symmetric' tensor product, that is, the quotient of the tensor product under the permutation of terms.

\subsection{Hypercommutative algebras and cohomological field theories}

\begin{definition}[Hypercommutative algebra]
A \emph{hypercommutative algebra} is an algebra over the operad 
$\HyperCom:=H_\bullet(\overline{\mathcal{M}}_{0, n+1})$ made up of the homology of the  Deligne--Mumford--Knudsen moduli spaces of stable genus~$0$ curves.  
\end{definition}
Such a structure is given by a morphism of operads $H_\bullet(\overline{\mathcal{M}}_{0, n+1}) \to \End_A$, and so amounts to a collection of symmetric multilinear maps $\mu_n\colon A^{\otimes n} \to A$ of degree $2(n-2)$ for each $n\ge 2$ that satisfy certain quadratic relations, see \cite{Manin99}. The first of those relations is the associativity of $\mu_2$, and further ones are higher associativity relations mixing operations together, hence the name ``hypercommutative''. 

The operad $H_\bullet(\overline{\mathcal{M}}_{0, n+1})$ is Koszul, with the Koszul dual cooperad $H_\bullet(\overline{\mathcal{M}}_{0, n+1})^{\ac}=H^{\bullet+1}({\mathcal{M}}_{0, n+1})$, the cohomology groups  of the moduli spaces of genus $0$ curves. So the operadic cobar construction  $$\Omega H^{\bullet+1}({\mathcal{M}}_{0, n+1})\stackrel{\sim}{\to} H_\bullet(\overline{\mathcal{M}}_{0, n+1})$$ provides a resolution of the former operad, see \cite{Getzler95}.

\begin{definition}[Homotopy hypercommutative algebras]
A  \emph{homotopy hypercommutative algebra} is an algebra over the operad $\Omega H^{\bullet+1}({\mathcal{M}}_{0, n+1})$. 
\end{definition}

This data amounts to an operadic twisting morphism $H^{\bullet+1}({\mathcal{M}}_{0, n+1}) \to \End_A$. The operations defining such a structure are parametrised by $H^{\bullet+1}({\mathcal{M}}_{0, n+1})$.  Hence, a homotopy hypercommutative algebra structure on a chain complex with trivial differential is made up of an infinite sequence of strata of multilinear operations, whose first stratum forms a hypercommutative algebra. 

\begin{definition}[Genus $0$ CohFT \cite{KontsevichManin94}]
Given a graded vector space $A$, a genus~$0$ \emph{cohomological field theory} (CohFT) on~$A$ is defined as a system of classes $\alpha_n\in H^\bullet(\overline{\calM}_{0,n+1})\otimes\End_A(n)$ satisfying the following properties.
\begin{itemize}
\item[$\diamond$] The classes $\alpha_n$ are equivariant with respect to the actions of the symmetric group $\mathbb{S}_n$ on the labels of marked points and on the factors of $\End_V(n)$.
\item[$\diamond$] The pullbacks via the natural mappings 
$\rho\colon \overline{\calM}_{0,n_1+1}\times \overline{\calM}_{0,n_2+1}\to \overline{\calM}_{0,n_1+n_2}$ 
produce the composition of the multilinear maps at the point corresponding to the preimage of the node  on the first curve:
 $$
\rho^*\alpha_{0,n_1+n_2}=\alpha_{0,n_1+1}\tilde{\circ}_i\, \alpha_{0,n_2+1}\ , 
 $$
where $\tilde{\circ}_i$ incorporates the composition in the endomorphism operad and the K\"unneth isomorphism. 
\end{itemize}  
\end{definition}

\begin{remark}
A CohFT is often required to have a unit $e_1\in A$; this corresponds to making use of the natural mappings $\pi\colon\overline{\calM}_{0,n+1}\to\overline{\calM}_{0,n}$. We shall not force that, and use all necessary formulae without the unit. Also, a CohFT in all genera needs $A$ to have a scalar product, and is defined using the language of modular operads. However, in genus~$0$, it is possible to eliminate it completely on the stage of applying the forgetful functor from modular operads to operads. The main advantage for doing so is to incorporate infinite dimensional spaces. An interested reader is referred to~\cite{DotsenkoShadrinVallette11,KhoroshkinMarkarianShadrin13} for details.  
\end{remark}

Summing up, the above definitions of a hypercommutative algebra  and of a genus~$0$ CohFT are the same. 

\subsection{Intersection theory on moduli spaces}\label{subsec:PsiClasses}

The Givental group action, we discuss below, makes use of the $\psi$-classes on moduli spaces of curves. 

\begin{definition}[$\psi$-classes]
Both the moduli space $\calM_{0,n}$ and its compactification $\overline{\calM}_{0,n}$ have $n$ tautological line bundles $\mathbb{L}_i$. The fibre of $\mathbb{L}_i$ over a point represented by a curve $C$ with marked points $x_1,\dots,x_n$ is equal to the cotangent line $T^*_{x_i}C$. The cohomology class $\psi_i$ of $\overline{\calM}_{0,n}$ is defined as the first Chern class of the line bundle~$\mathbb{L}_i$: $\psi_i=c_1(\mathbb{L}_i)\in H^2(\overline{\calM}_{0,n})$.
\end{definition}

{Recall that one can define the push-forward maps $\rho_*$ on the cohomology using the Poincar\'e duality  and the push-forward on the homology: 
\begin{eqnarray*}
& H^\bullet(\overline{\calM}_{0,n_1+1})\otimes H^\bullet(\overline{\calM}_{0,n_2+1})\to H^\bullet(\overline{\calM}_{0,n_1+1}\times\overline{\calM}_{0,n_2+1})\to 
H_{d-\bullet}(\overline{\calM}_{0,n_1+1}\times\overline{\calM}_{0,n_2+1}) &\\& \to
H_{d-\bullet}(\overline{\calM}_{0,n_1+n_2+1}) \to H^{\bullet+2}(\overline{\calM}_{0,n_1+n_2+1})\ ,& 
\end{eqnarray*}
where the dimension $d$ is equal to $+2n_1+2n_2-8$. Throughout the paper, we will only use the gluing along the point marked by $1$ on the first curve and the point marked by $0$ on the second one. }

{The main ingredients needed for computation with $\psi$-classes are the following formulae. They correspond to the expression the $\psi$-classes in terms of divisors, see e.~g.~\cite[\S VI.3.]{Manin99}.}
 
\begin{proposition}
{The Poincar\'e duals $\beta_n \in H^0(\overline{\calM}_{0,n+1})$  of the fundamental classes of the moduli spaces satisfy the following properties.}
{\begin{itemize}
 \item[$\diamond$] For all $i_1, i_2\in\underline{n}$,
\begin{equation}\label{TRR-0-root}
\psi_0=\sum_{\substack{I\sqcup J=\underline{n}\\ i_1,i_2\in I}} \rho_*(\beta_{|J|+1}\otimes\beta_{|I|})\ .
\end{equation}
 \item[$\diamond$] For all $i \in\underline{n}$,
\begin{equation}\label{TRR-0-sym}
\psi_i+\psi_0=\sum_{\substack{I\sqcup J=\underline{n}\\ i\in I}} \rho_*(\beta_{|J|+1}\otimes\beta_{|I|})\ .
\end{equation}
\end{itemize}}
\end{proposition}

\subsection{Givental  action on CohFTs}\label{subsec:GiventalAction}
 
In the case of genus~$0$ CohFTs, it is possible to extend the action of the Lie algebra~\cite{Lee2009} of the Givental group~\cite{Givental01, Givental01bis} to the Lie algebra $z\End(A)[[z]]$ dropping the assumption on (skew-)symmetry of the components of operators~\cite{KhoroshkinMarkarianShadrin13,Teleman12}. Let us recall the corresponding formulae, which we shall later identify from the homotopy viewpoint. For a genus 0 CohFT given by a system of classes $\alpha_n\in H^\bullet(\overline{\calM}_{0,n+1})\otimes\End_A(n)$, this action is defined by the formula
\begin{multline}\label{Givental-action-genus-0}
(\widehat{r_k z^{k}}.\{\alpha\})_n=
(-1)^{k+1}r_k\circ_1\alpha_{n}\cdot\psi_0^{k}+\sum_{m=1}^n\alpha_{n}\cdot\psi_m^{k}\circ_m r_k+\\+
\sum_{\substack{I\sqcup J=\underline{n}, |I|\ge2,\\ i+j=k-1}}(-1)^{i+1}\, 
\tilde{\rho}_*\left(\big(\alpha_{|J|+1}\cdot\psi_1^j\big)\otimes  \big(r_k \circ_1\alpha_{|I|}\cdot \psi_0^i\big)\right).
\end{multline}
{Here we assume that the output of every operadic element corresponds to the point marked by $0$ on the curve. In the last term, the map $\tilde{\rho}_*$ is defined by $\rho_*\otimes \circ_1$, i.e. an enrichment of the push-forward map on $H^\bullet(\overline{\calM}_{0,n+1})$ with the operadic composition on $\End_A$. }

\subsection{Trees}\label{subsec:ShTrees}

A \emph{reduced rooted tree} is a rooted tree whose vertices have at least one input. We consider the category of reduced rooted trees with leaves labelled bijectively from $1$ to $n$, denoted by $\Tree$. The trivial tree $|$ is considered to be part of $\Tree$.

A \emph{shuffle tree}, see \cite[$\S  2.8$]{Hoffbeck10} and \cite[$\S 3.1$]{DotsenkoKhoroshkin10},  is a reduced planar rooted tree equipped  satisfying the following condition. Suppose that we put labels on all edges by going down from the leaves to the root and labelling each edge by the minimum of the labels of the inputs of its top endpoint. Then, for each vertex, the labels of its inputs, read from left to right, should appear in the increasing order.

\begin{figure}[h]
$$\xymatrix@=1em{
&&&&&6\ar@{-}[dr]&&7\ar@{-}[dl]\\
1\ar@{-}[dr]&4\ar@{-}[d]&8\ar@{-}[dl]&&3\ar@{-}[dr]&5\ar@{-}[d]&*{}\ar@{-}[dl]^{6}&&\\
&*{}\ar@{-}[drr]_{1}&&2\ar@{-}[d]&&*{}\ar@{-}[dll]^3&&&\\
&&&*{}\ar@{-}[d]&&&\\
&&&&&&} $$
\caption{Example of a shuffle tree} \label{Fig:ShTree}
\label{fig:ShuffleTree}
\end{figure}
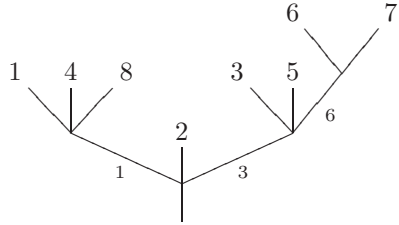

Shuffle trees provide us with choices of planar representatives for trees in space. In the sequel, we will need  shuffle binary trees, that we denote by $\mathsf{SBT}_n$. The labels of the leaves of a shuffle tree $t$, read from left to right, provide us with a permutation $\sigma^t$ of $\Sy_n$. In the example of Figure~\ref{fig:ShuffleTree}, this permutation is $\sigma^t=[14823567]$. \\

The underlying $\Sy$-module of the (conilpotent) cofree cooperad $\TTT^c(M)$ on an $\Sy$-module $M$ is given by the direct sum 
$\bigoplus_{t\in \Tree} t(M)$, where $t(M)$ is the treewise tensor module obtained by labelling every vertex of the tree $t$ with an element of $M$ according to the arity and the action of the symmetric groups. 
Its decomposition map is given by cutting the trees horizontally; see \cite[Chapter~$5$]{LodayVallette12} for more details.

The subcategory of trees with $n$ vertices is denoted by $\Tree^{(n)}$. The number of vertices endows the cofree cooperad $\TTT^c(M)\cong \bigoplus_{n \in \NN} \TTT^c(M)^{(n)}$ with a weight grading.

\subsection{Homotopy Lie algebra}

\begin{definition}[$L_\infty$-algebra]\label{DefPropL_infty}
An \emph{$L_\infty$-algebra} structure on a dg module $(A, d_A)$ is a family of totally skew-symmetric maps
$\ell_n \colon  A^{\otimes n} \to A$ of degree $|\ell_n|=n-2$, for all $n\ge 2$, satisfying the relations
$$\partial_A(\ell_n)=\sum_{\substack{p+q=n+1\\  p, q>1}} \sum_{\sigma\in Sh^{-1}_{p,q-1}} \sgn(\sigma) (-1)^{(p-1)q} (\ell_p \circ_1 \ell_q)^\sigma\ , \quad \text{for} \ n\ge 2\ ,$$
where $\partial_A$ is the differential of $\End_A$ induced by $d_A$ and where $Sh_{p,q-1}$ denotes the set of 
$(p,q-1)$-shuffles.
\end{definition}

For any shuffle binary tree $t$ with $n$ leaves, we consider its underlying planar binary tree $\bar{t}$ with $n-1$ vertices. To this planar binary tree, we associate a permutation of  $\Sy_{n-1}$  as follows. First, we put the vertices on $n-1$ distinct upward levels. 
This means that, among the trees with levels that represent $\bar t$, we choose the tree whose levels of the vertices, which are  at the same level in $\bar t$, go upward when moving from left to right. We label the levels by $\{1, \ldots, n-1\}$ from top to bottom and we label the vertices by $\{1, \ldots, n-1\}$ from left to right. The assignment which gives the level of each vertex defines a permutation $\sigma_{\bar t}$ of $\Sy_{n-1}$. 

\textsc{Example.}
$$
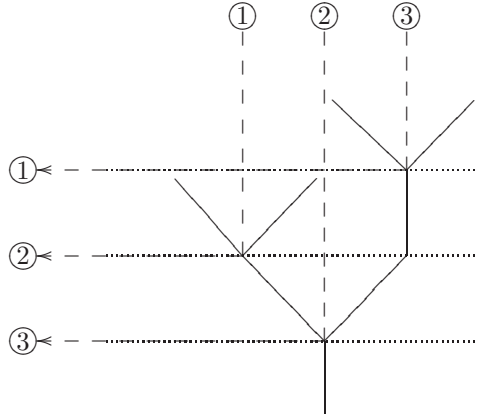
\xymatrix@R=20pt@C=20pt{
&&&*+[o][F-]{1} \ar@{--}[ddd] |!{[dd]}\hole &*+[o][F-]{2} \ar@{--}[dddd]|!{[dd]}\hole&*+[o][F-]{3} \ar@{--}[dd]&&& \\
&&&&&&&& \\
*+[o][F-]{1} & \ar@{..}[rrrrr]  &&&& *{}\ar@{-}[ul]\ar@{-}[ur]  \ar@{-->}[lllll] &&& \\
*+[o][F-]{2} &\ar@{..}[rrrrr]&&*{}\ar@{-}[ur]\ar@{-}[ul] \ar@{-->}[lll]   &&  *{}\ar@{-}[u] &&& \\
*+[o][F-]{3} &\ar@{..}[rrrrr]&&&*{}\ar@{-}[ul]*{}\ar@{-}[d] \ar@{-}[ur]  \ar@{-->}[llll]&    &&& \\
&&&&&&&& }
$$
In this example, the associated permutation is $\sigma_{\bar t}=[132]$.

\begin{theorem}[Homotopy Transfer Theorem, see \cite{LodayVallette12}]\label{Thm:TransferThm}
 Let $(V, d_V)$ be a homotopy retract of $(A, d_A)$:
\begin{eqnarray*}
&\xymatrix{     *{ \quad \ \  \quad (A, d_A)\ } \ar@<1ex>@(dl,ul)[]^{h}\ \ar@<0.5ex>[r]^{p} & *{\
(V,d_V)\quad \ \  \ \quad }  \ar@<0.5ex>[l]^{i}}&\\
& \mathrm{id}_A-i p =d_A  h
+ h  d_A,\  i\ \text{quasi-isomorphism}\ .
\end{eqnarray*}
Let the bracket $[ \, , ]\,:  \, A^{\otimes 2} \to A$ endow $A$ 
with a dg Lie algebra structure. 
The maps $\{ \ell_n\,:  \, V^{\otimes n} \to V\}_{n\ge 2}$ defined by 
$$\ell_n:=\sum_{t\in \mathsf{SBT}_n} \sgn(\sigma^t)\sgn(\sigma_{\bar t}) \  p \, t([ \; , \, ], h) \, i^{\otimes n}\ , $$
where the notation $t([ \, ,  ], h)$ stands for the $n$-multilinear operation on $A$ defined by the composition scheme $t$ with vertices labelled by $[ \, ,  ]$ and internal edges labelled by $h$, 
 define an $L_\infty$-algebra structure  on  $V$. 

Moreover, the maps $i_1:=i$ and 
$$i_n:=\sum_{t\in \mathsf{SBT}_n} \sgn(\sigma^t)\sgn(\sigma_{\bar t})   \ h \, t([ \, ,  ], h) \, i^{\otimes n}\ , \ \text{for}\ n\ge 2\ ,  $$
define an $\infty$-quasi-isomorphism from the transferred $L_\infty$-algebra $(V, d_V, \{  \ell_n\}_{n\ge 2})$ to 
the dg Lie algebra $(A, d_A, [ \, ,  ])$. 
\end{theorem}

\begin{proof} 
Let us make explicit the signs in the proof of  \cite[Theorem~$10.3.3$]{LodayVallette12}. First one easily checks that the map $\psi$ yields no sign since one starts from a degree $0$ Lie bracket. (The signs coming from the permutations of the suspensions $s$ and the homotopy $h$ cancel.) The map $s\varphi$ has degree $2$ and so produces no sign. Therefore, the only sign is the one coming from the  decomposition map of the  cooperad $\mathop{\mathrm{Lie}}^{\ac}=\End^c_{\KK s^{-1}}\otimes_{H}\Com^*$. The  decomposition map  of the cooperad $\Com^*$ is given by the sum of all the binary trees, that we choose to represent with shuffle trees. And the decomposition map of the cooperad $\End^c_{\KK s^{-1}}$ is given by the sum of all the shuffle binary trees $t$  with coefficient exactly $\sgn(\sigma^t)\sgn(\sigma_{\bar t})$.
\end{proof}

\subsection{Convolution  algebras}\label{sec:ConvAlg}

Let $\C$ be a dg cooperad and let $\P$ be a dg operad.
Recall that the collection $\Hom(\C, \P):=\{ \Hom(\C(n), \P(n))\}_{n\in \NN}$ forms an operad called the \emph{convolution operad}, see \cite{LodayVallette12}. This structure induces a dg pre-Lie algebra structure and hence a dg Lie algebra structure on equivariant maps 
$$\Hom_\Sy(\C, \P):=\big(\prod_{n\in \NN} \Hom_{\Sy_n}({\C}(n), \P(n)), \partial, [\;,\,] \big)\ . $$
Explicitly, the Lie bracket is given by 
$$[f,g]:=   {\gamma}_{\P} \circ  \big(f\otimes g - (-1)^{|f||g|}g\otimes f \big) \circ    \Delta_{(1)}\ ,$$
where $\Delta_{(1)} : \C \to \TTT(\C)^{(2)}$ is the partial decomposition map of the cooperad $\C$.
In this \emph{convolution dg Lie algebra}, we consider the Maurer--Cartan equation 
$$\partial (\alpha) + \textstyle{\frac{1}{2}}[\alpha, \alpha]=0\ ,$$ whose degree $-1$ solutions are called \emph{twisting morphisms} and  denoted $\Tw(\C, \P)$.

All the  dg Lie algebras of this paper are of this form, where the cooperad is the Koszul dual dg cooperad $\P^{\ac}$ of an operad $\P$ and where the operad is the endomorphism operad $\End_A$: 
$$\g_\P:=\big(\prod_{n\in \NN}\Hom_{\Sy}(\P^{\ac}(n), \End_A(n)), \partial:=(\partial_A)_* - (d_{\P^{\ac}})^*,  [\;,\,]     \big) \  .$$

\begin{theorem}[``Rosetta Stone'', see \cite{LodayVallette12}]\label{thm:4 def theo}
The set of homotopy $\P$-algebra structures on a dg module $A$ is equal to
\begin{align*}
\Hom_{\mathsf{dgOp}}\left(\Omega \P^{\ac}, \End_A\right) \cong
\Tw(\P^{\ac}, \End_A)& \cong \mathrm{Codiff}(\P^{\ac}(A))\ .
\end{align*}
\end{theorem}

There is a generalisation of cooperads, where the decomposition coproduct is relaxed up to homotopy; this algebraic structure is called a \emph{homotopy cooperad}. 
It is made up of a dg $\Sy$-module $\C$ equipped with a decomposition map
$\Delta_{\C}\colon\C\to\TTT(\C)^{(\ge 2)}$, which satisfies some relation, see \cite{MerkulovVallette09I}. This notion is dual to the notion of homotopy operad of \cite{VanderLaan02}.

When $(\C, \Delta_\C)$ is a homotopy cooperad and $(\P, \gamma_\P)$ is a dg operad with the induced composition map
$\widetilde{\gamma}_{\P}\colon \allowbreak\TTT(\P)^{(\ge 2)}\to\P$,  the collection $\Hom(\mathcal{C},\mathcal{P})$ is a homotopy operad, called the \emph{convolution homotopy operad}. The direct product of $\Sy$-invariants of this collection 
$$\Hom_\Sy(\C, \P):=\big(\prod_{n\in \NN} \Hom_{\Sy_n}({\C}(n), \P(n)), \partial, \{\ell_n\}_{n\ge 2} \big)$$
is an $L_\infty$-algebra~\cite{VanderLaan02}. This algebra is referred to as the \emph{convolution $L_\infty$-algebra}; its structure maps $\ell_n$ are given by the formula
\begin{equation}\label{l-infty}
\ell_n(f_1,\ldots,f_n)=\sum_{\sigma\in S_n}(-1)^{\sgn(\sigma,f_1,\ldots,f_n)}\widetilde{\gamma}_{\P}\circ(f_{\sigma(1)}\otimes\cdots\otimes f_{\sigma(n)})\circ\Delta^{(n)}_{\C},
\end{equation}
where $\Delta^{(n)}_{\C}$ is the component of ${\Delta}_{\C}$ which maps $\C$ to $\TTT(\C)^{(n)}$, see~\cite{MerkulovVallette09I,VanderLaan02}. 

In such an algebra, we can consider the (generalised) Maurer--Cartan equation 
$$\sum_{n \ge 1} \frac{1}{n!} \ell_n(\alpha, \ldots, \alpha) =0\ ,$$
whose degree $-1$ solutions are called (generalised) \emph{twisting morphisms} and  denoted by $\Tw(\C, \P)$.
Notice that this equation, as well as other formulae throughout this paper, makes sense for homotopy convolution algebras, since for every element $c\in \C$, its image under the decomposition map $\Delta_\C$ is a finite sum. 

The data of a homotopy cooperad $\C$ is equivalent to the data of a quasi-free dg operad structure $\Omega_\infty \C$ on $\mathcal T (s^{-1}\C)$. Generalised twisting morphisms  are then in one-to-one correspondence with morphisms of dg operads from $\Omega_\infty \C$ to $\P$: 
$$\Hom_{\mathsf{dgOp}}\left(\Omega_\infty  \C, \P\right)\cong \Tw(\C, \P) \ .$$

\subsection{Gauge symmetries in homotopy Lie  algebras}\label{subsec:GaugeRecollections}

Let $\big(\mathfrak{l}, \{\ell_n\}_{n\ge 1} \big)$ be an $L_\infty$-algebra, and let $\alpha$ be a (generalised) Maurer--Cartan element of that algebra. One can twist  the original structure maps of~$\mathfrak{l}$ with $\alpha$: $$\ell^\alpha_n(x_1,\ldots,x_n):=\sum_{p\ge 0} \frac{1}{p!} \ell_{n+p}(\underbrace{\alpha, \ldots, \alpha}_p, x_1,\ldots, x_n)\ , $$ 
so that $\mathfrak{l}^\alpha:=\big(\mathfrak{l}, \{\ell^\alpha_n\}_{n\ge 1} \big)$ forms again an $L_\infty$-algebra, called a \emph{twisted $L_\infty$-algebra}. 
Recall that a degree $-1$ element $\tau\in\mathfrak{l}$ is an infinitesimal deformation of $\alpha$, i.e. $$\alpha+\varepsilon \tau \in \MC\big(\mathfrak{l}\otimes \KK[\varepsilon]/(\varepsilon^2)\big)\ , $$ if and only if $\ell^\alpha_1(\tau)=0$. So the tangent space of the Maurer--Cartan variety at the point $\alpha$ is equal to 
$$T_\alpha\,  \MC(\mathfrak{l}) = \Ker\,  \ell^\alpha_1\ . $$ 
In particular, if $\lambda$ is an element of degree~$0$ in $\mathfrak{l}$, the element $\tau_\lambda=\ell^\alpha_1(\lambda)$  satisfies the equation $\ell^\alpha_1(\tau_\lambda)=0$; so such an element defines an infinitesimal deformation of $\alpha$. The element $\tau_\lambda$ depends on $\alpha$, and so defines a vector field; we just checked that this vector field is a tangent vector field of the Maurer--Cartan variety. Its integral curves give deformations of Maurer--Cartan elements, and define \emph{gauge symmetries} of Maurer--Cartan elements of a $L_\infty$-algebra, see \cite{Getzler09}.

\section{The homotopy Lie algebra encoding skeletal homotopy BV-algebras}\label{sec:HoLie-BV}

In this section, we recall the notion of Batalin--Vilkovisky (BV) algebras and its  homotopy version. 
We develop the deformation theory of homotopy BV-algebras with a particular convolution $L_\infty$-algebra inside which we shall be able to describe the infinitesimal Givental  action in the next section. 

\subsection{Homotopy BV-algebras and skeletal homotopy BV-algebras}\label{subsec:HoBV}

This section is a brief summary of constructions and results of~\cite{GCTV12} and \cite{DrummondColeVallette11} that we use.

\begin{definition}[dg $\BV$-algebra]
A \emph{dg $\BV$-algebra} $(A, d,  \bullet, \Delta)$ is a differential graded commutative algebra equipped with a square-zero degree $1$ operator $\Delta$ of order at most~$2$.
\end{definition}

Notice that any $\BV$-algebra includes a degree $1$ Lie bracket $\langle\,\textrm{-} , \textrm{-}\,\rangle$ defined by 
$$\langle\,\textrm{-} , \textrm{-}\,\rangle\ =\ \Delta (\textrm{-} \bullet \textrm{-})\ -\
(\Delta(\textrm{-}) \bullet \textrm{-})   \ - \ (\textrm{-} \bullet
\Delta(\textrm{-})) \ .$$ 
This induces a quadratic-linear presentation $\TTT(\bullet, \Delta, \langle\,\textrm{-} , \textrm{-}\,\rangle)/(R)$  for the operad $\BV$ encoding $\BV$-algebras, see \cite[Section~$1$]{GCTV12} for a complete exposition.
Its  Koszul dual dg cooperad was proved to be equal to 
$$\BV^{\ac}\cong (G^{\ac}[\delta], d_\varphi)\ , $$ 
where $G^{\ac}$ stands for the Koszul dual cooperad of the operad $G$ encoding Gerstenhaber algebras, where $\delta:=s \Delta$ is a degree $2$ element of arity $1$ and where $d_\varphi$ is the unique coderivation extending 
$$ \vcenter{\xymatrix@M=2pt@R=8pt@C=8pt{
 \ar@{-}[dr] & &\ar@{-}[dl]   \\
 &  {\scriptstyle s\bullet }     \ar@{-}[d] &  \\
& {\scriptstyle s\Delta} & }} -
\vcenter{\xymatrix@M=2pt@R=8pt@C=8pt{
  \ar@{-}[dr]{ \scriptstyle s\Delta} & &\ar@{-}[dl]   \\
 &  {\scriptstyle s\bullet }     \ar@{-}[d] &  \\
&  & }} -
\vcenter{\xymatrix@M=2pt@R=8pt@C=8pt{
  \ar@{-}[dr]& & {\scriptstyle s\Delta} \ar@{-}[dl]   \\
 &  {\scriptstyle s\bullet }     \ar@{-}[d] &  \\
&  & }}  \quad\mapsto\quad
\vcenter{\xymatrix@M=2pt@R=8pt@C=8pt{
 \ar@{-}[dr] & &\ar@{-}[dl]   \\
 &  {\scriptstyle s\la\; , \, \ra}     \ar@{-}[d] &  \\
& & . }} $$

\begin{theorem}\cite[Theorem~$6$]{GCTV12}
The operad $\BV$ is a nonhomogeneous Koszul operad, i.e. the cobar construction of $\BV^{\ac}$ is a resolution of $\BV$: 
$$\BV_\infty:=\Omega \BV^{\ac} \ \qi \ \BV\ . $$
\end{theorem}

Algebras over the Koszul resolution $\BV_\infty$ are called \emph{homotopy BV-algebras}. This resolution is already much smaller than the bar-cobar resolution but is not minimal. Let us explain, following~\cite{DrummondColeVallette11}, how to derive the minimal resolution from it.

We consider the $\Sy$-module  $M$ made up of the two elements $\mu$ and $\beta$, both of arity two with trivial symmetric group action, in degrees $1$ and $2$ respectively: 
$$ M:=\KK_2\underbrace{ s\bullet}_\mu\oplus \KK_2\underbrace{s\la \; , \, \ra}_{\beta}  \ .$$
Let $\psi$ denote the degree one morphism of graded $\Sy$-modules $\psi:\TTT^c(M)\to M$ which first projects $\TTT^c(M)$ to the cogenerators $M$ and then takes $\mu$ to $\beta$ and $\beta$ to zero. The map  $\psi$ extends uniquely to a degree one coderivation $d_\psi$ of $\TTT^c(M)$, which amounts to applying $\psi$ everywhere. So its image is equal to the sum over the vertices labelled $\mu$ of trees where this $\mu$ is changed for a $\beta$. 

The Koszul dual cooperad $G^{\ac}$ is a sub-cooperad of the cofree cooperad $\TTT^c(M)$ and the coderivation $d_\varphi$ of $G^{\ac}[\delta]$ is equal to $\delta^{-1}d_\psi$.

Let $t$ be a binary tree, that is a tree where all the vertices have total valence $3$.  Any vertex $v$ has some number of leaves $m_v$ above one of its incoming edges, and another number $n_v$ above the other.  Let the \emph{weight} $\omega(v)$ be their product $m_vn_v$. The sum of the weights of all the vertices of a binary tree with $n$ leaves is equal to~$\binom{n}{2}$.

\begin{definition}[The map $H$]
Let $\bar{H}:M\to M$ be the degree $-1$ morphism of graded $\Sy$-modules given by sending $\beta$ to $\mu$ and $\mu$ to $0$.  
We define the map $H$  on a decorated tree with $n$ leaves in $\TTT^c(M)$ as a sum over the vertices.  For the vertex $v$, the contribution to the sum is $\frac{\omega(v)}{\binom{n}{2}}$ times the decorated tree obtained by applying $\bar{H}$ to $v$, including the Koszul sign.  
\end{definition}
So the map $H$ has a similar flavour to extending $\bar{H}$ as a coderivation, but also includes combinatorial factors.

\textsc{Example.}
The image of 
$$\vcenter
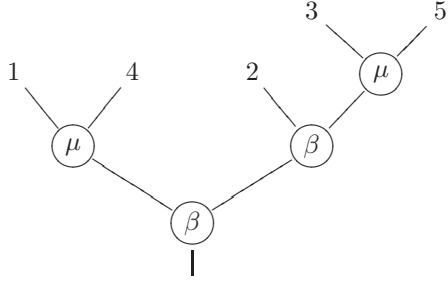
{\xymatrix@R=8pt@C=8pt{
 & & & & & 3 \ar@{-}[dr]& &5 \ar@{-}[dl] \\
1\ar@{-}[dr]  & &4\ar@{-}[dl]  & &2\ar@{-}[dr] & &*++[o][F-]{\mu} \ar@{-}[dl] & \\
   & *++[o][F-]{\mu} \ar@{-}[drr]& & & & *++[o][F-]{\beta} \ar@{-}[dll]& & \\
    & & & *++[o][F-]{\beta} \ar@{-}[d]& & & & \\
     & & & & & & & }} $$ 
under the map $H$ is equal to 
$$\frac{3}{5}\vcenter{\xymatrix@R=8pt@C=8pt{
 & & & & & 3 \ar@{-}[dr]& &5 \ar@{-}[dl] \\
1\ar@{-}[dr]  & &4\ar@{-}[dl]  & &2\ar@{-}[dr] & &*++[o][F-]{\mu} \ar@{-}[dl] & \\
   & *++[o][F-]{\mu} \ar@{-}[drr]& & & & *++[o][F-]{\beta} \ar@{-}[dll]& & \\
    & & & *++[o][F-]{\mu} \ar@{-}[d]& & & & \\
     & & & & & & & }} 
     -\frac{1}{5}
     \vcenter{\xymatrix@R=8pt@C=8pt{
 & & & & & 3 \ar@{-}[dr]& &5 \ar@{-}[dl] \\
1\ar@{-}[dr]  & &4\ar@{-}[dl]  & &2\ar@{-}[dr] & &*++[o][F-]{\mu} \ar@{-}[dl] & \\
   & *++[o][F-]{\mu} \ar@{-}[drr]& & & & *++[o][F-]{\mu} \ar@{-}[dll]& & \\
    & & & *++[o][F-]{\beta} \ar@{-}[d]& & & & \\
     & & & & & & & }} \ .$$ 

\begin{proposition}[\cite{DrummondColeVallette11}]\label{prop:MainDefRetract}
The maps $H$ and $d_\psi$ defined on $\TTT^c(M)$ restrict to the sub-cooperad $G^{\ac}\subset \TTT^c(M)$, and give rise to the following deformation retract:
\begin{eqnarray*}
\xymatrix@C=30pt{     *{
\big({G}^{\ac}[\delta], d_\varphi= \delta^{-1} d_\psi \big)  \  \ } \ar@<9ex>@(dl,ul)[]^{\delta H}\ \ar@{->>}@<0.5ex>[r]^(0.45){pr} & *{\ \ \ 
\big(\, {T}^c(\delta)   \oplus \Im\, Hd_\psi, 0\big), \quad \ \  \ \quad }  \ar@{>->}@<0.5ex>[l] }
\end{eqnarray*}
where $pr$ is the sum of the projection onto ${T}^c(\delta)$ in non-negative $\delta$-degrees and the projection $Hd_\psi$ in $\delta$-degree~$0$. 
\end{proposition}

The right-hand side computes the Quillen homology $H^Q(\BV)$ of the operad $\BV$, i.e. the homology of the bar construction of $\BV$. We shall denote it by $\mathcal{H}\oplus \I$ for brevity. It  can be expressed in terms of the cohomology of the moduli space of curves of genus~$0$: 
$$\mathcal{H} :=\overline{H}^Q(\BV)\cong \bar{T}^c(\delta)   \oplus \Im\, Hd_\psi \cong \bar{T}^c(\delta) \oplus H^{\bullet+1}(\mathcal{M}_{0,n+1})\ . $$
In \cite{DrummondColeVallette11}, the Homotopy Transfer Theorem for homotopy cooperads was used to transfer the dg  cooperad structure of ${\BV}^{\ac}$ to a homotopy cooperad structure on 
${\mathcal{H}}$ 
via the above deformation retract. 
In operadic terms, the short exact sequence of homotopy cooperads 
$$\bar{T}^c(\delta) \mono {\mathcal{H}} \epi  H^{\bullet+1}(\mathcal{M}_{0,n+1})$$ 
is  exact, i.e. ${\mathcal{H}}$ is an extension of the (non-unital) cooperads 
 $$\bar{T}^c(\delta)=H^\bullet(S^1)^{\ac}\quad\text{ and }\quad H^{\bullet+1}(\mathcal{M}_{0,n+1})=\allowbreak H_\bullet(\overline{\mathcal{M}}_{0,n+1})^{\ac}.$$

\begin{theorem}[\cite{DrummondColeVallette11}]\label{thm:sBVinfinity}
The cobar construction of the homotopy cooperad ${\mathcal{H}}$ is the minimal model of the operad~$\BV$:
$$\sBV_\infty:=\Omega_\infty\,  {\mathcal{H}} \qi \BV. $$
\end{theorem}
Algebras over the minimal model $\sBV_\infty$ are called \emph{skeletal homotopy BV-algebras}.

\subsection{The three convolution dg Lie algebras}
The general construction of convolution dg Lie algebras from Section~\ref{sec:ConvAlg} can be used to produce three dg Lie algebras that we shall use in this paper. 
Applying the general construction to the dg cooperad $\C={\BV}^{\ac}$, we obtain the convolution dg Lie algebra $\g_{\BV}$ that encodes  homotopy $\BV$-algebras structures. The formulae of Section~\ref{subsec:HoBV} show that $$\g_{\BV}\cong ( \g_{G}[[z]], z (d_\psi)^*),$$ where $z$ is a degree $-2$ element.
Applying the general construction to the (non-unital) cooperad $\C=\overline{T}^c(\delta)$, we obtain the convolution dg Lie algebra
$$\g_\Delta :=\big(\Hom(\bar{T}^c(\delta), \End(A)), (\partial_A)_* \big)\cong
\big(z\End(A)[[z]], \partial_A\big)$$  
that encodes multicomplex structures, see \cite{DotsenkoShadrinVallette12}. Notice that this dg Lie algebra is equal to Givental dg Lie algebra of Section~\ref{subsec:GiventalAction}.
Finally, applying the general construction to the  cooperad $\C=\HyperCom^{\ac}$, we obtain the convolution dg Lie algebra
$$\g_{\HyperCom} :=\big(\Hom_\Sy(H^{\bullet+1}(\mathcal{M}_{0,n+1}), \End_A), (\partial_A)_* \big)\cong
\big(\Hom_\Sy(\Im\, Hd_\psi, \End_A), (\partial_A)_* \big)\ ,$$ 
which encodes homotopy hypercommutative algebra structures.

\subsection{The convolution homotopy Lie algebra of skeletal homotopy BV-algebras}\label{subsec:Linfty}
Applying the general construction of convolution homotopy Lie algebras to the homotopy cooperad $\C={\mathcal{H}}$, we obtain the convolution $L_\infty$-algebra 
 $$
\mathfrak{l}_{\BV}:=\big(\prod_{n\ge 1} \Hom_{\Sy_n}{\mathcal{H}}(n), \End_A(n)), (\partial_A)_*, \{\ell_n\}_{n\ge 2}\big).
 $$ 

\begin{proposition}[\cite{DrummondColeVallette11}]\label{prop:ReducedHBV-TW}
The set of skeletal homotopy $\BV$-algebra structures on a dg module $A$ is equal to
$$ \Hom_{\mathsf{dgOp}}(\sBV_\infty, \End_A)\cong  \Tw({\mathcal{H}}, \End_A) \ .$$
\end{proposition}

 The combinatorics of the homotopy transfer theorem for cooperads
 makes the computations in the convolution $L_\infty$-algebra $\mathfrak{l}_{\BV}$  unfeasible. To make them manageable, we go the other way round, first considering the convolution Lie algebra $\g_{\BV}$ and then applying the homotopy transfer theorem for homotopy Lie algebras (Theorem~\ref{Thm:TransferThm}).

\begin{proposition}\label{prop:HTTCoopLinfty}
The above $L_\infty$-algebra structure $\mathfrak{l}_{\BV}$ on $\Hom_\Sy({\mathcal{H}}, \End_A)$ is isomorphic to the $L_\infty$-algebra structure obtained by transferring the dg Lie algebra structure of $\g_{\BV}$
under the formulae of Theorem~\ref{Thm:TransferThm} and the following deformation retract
\begin{eqnarray*}
\xymatrix@C=30pt{     
*{\big(\Hom_{\Sy}(\overline{\BV}^{\ac}, \End_A), \partial\big)  \  \ } 
\ar@<11ex>@(dl,ul)[]^{(\delta H)^*}\ \ar@{->>}@<0.5ex>[r] & 
*{\ \ \ \big(\Hom_\Sy({\mathcal{H}}, \End_A), (\partial_A)_* \big), \ }  \ar@{>->}@<0.5ex>[l]^{pr^*}
}
\end{eqnarray*}
\end{proposition}

\begin{proof} This proposition follows from the following general result: the formulae for the transferred homotopy cooperad  \cite[Theorem~$3.3$]{DrummondColeVallette11} and for the transferred homotopy Lie algebra (Theorem~\ref{Thm:TransferThm}) 
 commute under the convolution homotopy Lie algebra functor
$$\Hom_{\Sy}(-, \P)   \ : \ \text{\sf homotopy cooperads} \to \text{\sf homotopy Lie algebras}   \  , $$ 
given in Section~\ref{sec:ConvAlg}.
Let $(\C, \Delta, d_\C)$ be a coaugmented dg cooperad and let $(\P, \gamma, d_\P)$ be a dg operad. Writing the underlying homology groups ${\mathcal{H}}$ as a deformation retract of $(\overline{\C}, d_\C)$
\begin{eqnarray*}
&\xymatrix{     *{ \quad \ \  \quad (\overline{\C}, d_\C)\ } \ar@(dl,ul)[]^{\eta}\ \ar@<0.5ex>[r]^-{\pi} & *{\
({\mathcal{H}}, 0) \quad \ \  \ \quad }  \ar@<0.5ex>[l]^-{\iota}} &
\end{eqnarray*}
allows one to transfer a homotopy cooperad structure as follows.  For any tree $t\in \Tree$ with at least $2$ vertices, we consider all the possible ways of writing it by successive substitutions of trees with two vertices:
$$t=(((t_1\circ_{j_1} t_2)\circ_{j_2} t_3) \cdots )\circ_{j_k} t_{k+1}    \ ,$$
where $t\circ_j s$ stands for the substitution of the tree $s$ at the $j^\textrm{th}$ vertex of $t$. 
The transferred structure map $\widetilde{\Delta}_t :  {\mathcal{H}}\to t({\mathcal{H}})$, for $t\in \textsf{Tree}$, is then given  by
$$
\widetilde{\Delta}_t:=\sum \pm\,   t(\pi)\circ \big(
(\Delta_{t_{k+1}} \eta)     \circ_{j_k}   (\cdots   (\Delta_{t_3} \eta)     \circ_{j_2}  ((\Delta_{t_2} \eta)  \circ_{j_1} \Delta_{t_1} )   )
     \big)\circ \iota \ , $$
where the notation $(\Delta_{t'} \eta)  \circ_j \Delta_{t}$ means  the composite of $\Delta_t$ with $\Delta_{t'} \eta$ at the $j^\textrm{th}$ vertex of the tree $t$. The induced $L_\infty$-algebra structure on the convolution algebra
$\Hom_\Sy({\mathcal{H}}, \P)$ is then equals to
$$ \ell_n(f_1, \ldots, f_n)\  : \ {\mathcal{H}}\xrightarrow{\iota} \overline{\C}  \to \TTT(\overline{\C})^{(n)}
 \xrightarrow{\TTT(\pi)} \TTT({\cH})^{(n)}  \xrightarrow{} \TTT(\P)^{(n)}
  \xrightarrow{\gamma} \P \ , $$
where the map $\TTT({\cH})^{(n)}  \to \TTT(\P)^{(n)}$ is
$\sum_{\sigma \in \Sy_n}  \sgn(\sigma) \ \TTT(f_{\sigma(1)}, \ldots, f_{\sigma(n)})$ and
where the map from $\overline{\C}$ to $\TTT(\overline{\C})^{(n)}$
is
$$\sum_{t\in \Tree^{(n)}} \sum \pm\,    \big(
(\Delta_{t_{n-1}} \eta)     \circ_{j_{n-2}}   (\cdots   (\Delta_{t_3} \eta)     \circ_{j_2}  ((\Delta_{t_2} \eta)  \circ_{j_1} \Delta_{t_1} )   )
     \big).$$ This latter map is equal to the iteration of the infinitesimal decomposition map of the cooperad $\overline{\C}$
$$\sum_{j_k\in \{1, \ldots, k+1   \}  } \sgn(\sigma_j)\,    \big(
(\Delta_{(1)} \eta)     \circ_{j_{n-2}}   (\cdots   (\Delta_{(1)} \eta)     \circ_{j_2}  ((\Delta_{(1)} \eta)  \circ_{j_1} \Delta_{(1)} )   )
     \big)
         \ ,$$
where the sign $\sgn(\sigma_j)$ is given by the permutation associated to the following planar binary tree $j$ with levels:
any sequence of integers $(j_1, \ldots, j_{n-2})$ gives rise to a  with $n-1$ vertices such that the binary vertex at level $n-1-k$ is at place $j_k$. All the other signs are straightforward applications of the sign rule of the permutations of graded elements.

On the other hand, the transferred $L_\infty$-algebra structure on $\Hom_\Sy({\mathcal{H}}, \P)$ through the pulled-back deformation retract
\begin{eqnarray*}
&\xymatrix@C=40pt{     *{ \quad\quad\quad\ \ \quad \ \  \quad (\Hom_\Sy(\overline{\C}, \P), \partial)\ } \ar@(dl,ul)[]^{h:=\eta^*}\ \ar@<0.5ex>[r]^-{p:=\iota^*} & *{\
(\Hom_\Sy({\mathcal{H}}, \P), \partial) \quad \ \  \ \quad }  \ar@<0.5ex>[l]^-{i:=\pi^*}} &
\end{eqnarray*}
given by Theorem~\ref{Thm:TransferThm} is
$$l_n:=\sum_{t\in \mathsf{SBT}_n} \sgn(\sigma^t)\sgn(\sigma_{\bar t}) \  i \, t([ \; , \, ], h) \, (p)^{\otimes n}\ , $$
where $[ \; , \, ]$ is the  bracket of the convolution Lie algebra $\Hom_\Sy(\overline{\C}, \P)$ equal to
$$[f, g]\ : \ \overline{\C} \xrightarrow{\Delta_{(1)}} \TTT(\overline{\C})^{(2)}
 \xrightarrow{\TTT(f,g)-\TTT(g,f)} \TTT(P)^{(2)}
  \xrightarrow{\gamma} \P \ . $$
So the map  given by the labelled trees $t([ \; , \, ], h)$ amounts to splitting the elements of $\overline{\C}$ in all possible ways via iterations of $(\Delta_{(1)}h)$. In the end, the map $l_n(f_1, \ldots, f_n)$ is equal to
$$l_n(f_1, \ldots, f_n)\  : \ {\mathcal{H}}\xrightarrow{\iota} \overline{\C}  \to \TTT(\overline{\C})^{(n)}
  \xrightarrow{} \TTT(\P)^{(n)}
  \xrightarrow{\gamma} \P \ , $$
where the map $\TTT(\overline{\C})^{(n)}  \to \TTT(\P)^{(n)}$ is
$\sum_{\sigma \in \Sy_n}  \sgn(\sigma) \ \TTT(f_{\sigma(1)}\pi, \ldots, f_{\sigma(n)}\pi)$ and
where the map from $\overline{\C}$ to $\TTT(\overline{\C})^{(n)}$
is
$\sum_{j_k\in \{1, \ldots, k+1   \}  } \sgn(\sigma_j)\,    \big(
(\Delta_{(1)} \eta)     \circ_{j_{n-2}}   (\cdots   (\Delta_{(1)} \eta)     \circ_{j_2}  ((\Delta_{(1)} \eta)  \circ_{j_1} \Delta_{(1)} )   )
     \big)$. The sign $\sgn(\sigma_{\bar t})$ coincides with the sign $\sgn(\sigma_j)$, the other signs are direct consequences of permutation of graded elements.  Therefore, $\ell_n=l_n$.
\end{proof}

The $L_\infty$-algebra  $\mathfrak{l}_{\BV}$ whose underlying space satisfies 
$$\mathfrak{l}_{\BV}=\g_{\Delta} \oplus \g_{\HyperCom}$$ is an extension of the two dg Lie algebras $\g_{\Delta}$ and $\g_{\HyperCom}$ in the category of $L_\infty$-algebras.

\section{Gauge interpretation of infinitesimal Givental action} \label{sec:GaugeInterpretation}

In this section, we use the $L_\infty$-algebra $\mathfrak{l}_{\BV}$ to identify the Givental action with the homotopy BV gauge symmetries. 

\subsection{The main theorem}
Any Maurer--Cartan element $\alpha$ in $\g_{\HyperCom}$, representing a homotopy hypercommutative algebra, is also a Maurer-Cartan element in  $\mathfrak{l}_{\BV}$. We shall deform $\alpha$ in the direction of $\g_{\Delta}$ in the $L_\infty$-algebra $\mathfrak{l}_{\BV}$. Let $r(z)=\sum_{l\ge 1} r_l z^l$ be a degree $0$ element of $\g_{\Delta}$. The general definition of gauge symmetries implies that $\ell^\alpha_1(r(z))$ is an infinitesimal deformation of $\alpha$. Note that although in general gauge symmetries for $L_\infty$-algebras form an $\infty$-groupoid~\cite{Getzler09}, in our particular case $\g_{\Delta}$ is a dg Lie subalgebra, so the respective symmetries form a group. A hypercommutative algebra can also be viewed as a genus~$0$ CohFT, and as such can be deformed by the infinitesimal Givental action. In this section, we show that these two deformations are exactly the same, proving the following result.

\begin{theorem}\label{thm:MainGiv=LinftyAction}
For any hypercommutative algebra structure on $A$ encoded by a Maurer--Cartan element $\alpha\in\g_{\HyperCom}$ and for any degree $0$ element $r(z)\in\g_{\Delta}$, the Givental action of $r(z)$ on $\alpha$ is equal to the gauge symmetry action of $r(z)$ on $\alpha$ viewed as a homotopy $\BV$-algebra structure:
$$\widehat{r(z)}.\alpha=\ell^\alpha_1(r(z))\ . $$
\end{theorem}
\noindent
The rest of this section is devoted to the proof of this formula.

\subsection{Gauge symmetries restrict to hypercommutative algebras}\label{subsec:GaugetoHyperCom}

Let us first check that the formula of Theorem~\ref{thm:MainGiv=LinftyAction} makes sense by proving the following lemma. 

\begin{lemma}\label{lem:THCom}
The infinitesimal gauge symmetry action of $r(z)$ on $\alpha$ deforms it in the class of hypercommutative algebra structures, {i.e. 
$$\ell^\alpha_1(r(z)) \in T_\alpha\,  \MC(\g_{\HyperCom})\ .  $$ }
\end{lemma}

\begin{proof}
Since neither $A$ nor the homotopy cooperad ${\mathcal{H}}$ 
has a differential, the first term in the formula 
$$\ell^\alpha_1(r(z))=\sum_{n\ge1}\frac{1}{(n-1)!}\ell_n(\alpha, \ldots, \alpha,r(z))$$
for the infinitesimal gauge symmetry vanishes. Therefore, only the terms with $n>1$ contribute to the above sum. Since the $L_\infty$-algebra $\mathfrak{l}_{\BV}$ is a convolution algebra, it is graded by arity of the maps minus one. So the map $\ell^\alpha_1(r(z))$ vanishes on the arity one elements $\overline{T}^c(\delta)$.

 Since $\alpha$ is a Maurer--Cartan element, any tree in the homotopy transfer formulae which has two leaves with the same parent both labelled by~$i(\alpha)$ contributes zero to the terms~$\ell_n(\alpha, \ldots, \alpha,r(z))$. Therefore the only non-trivial contributions to the homotopy transfer formulae for $$\ell_n(\alpha, \ldots, \alpha,r(z))=(-1)^{n-1}\ell_n(r(z),\alpha, \ldots, \alpha)$$ are given by ``left combs'' 
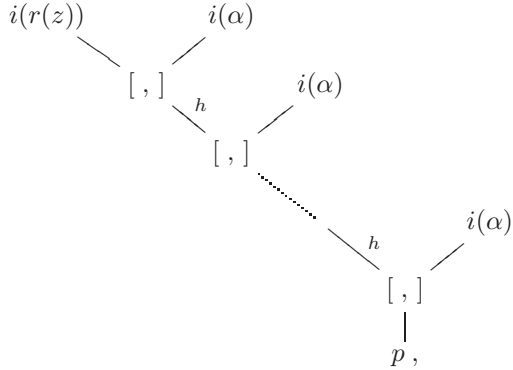
$$\xymatrix@=1em{i(r(z)) \ar@{-}[dr] & &i(\alpha)\ar@{-}[dl] & && \\
&\ar@{-}[dr]^h [ \; , \, ]&&i(\alpha)  \ar@{-}[dl]& &\\
&&\ar@{..}[dr][ \; , \, ]&& &\\
&&&\ar@{-}[dr]^h& &i(\alpha)  \ar@{-}[dl]\\
&&&&[ \; , \, ]\ar@{-}[d]& \\
&&&&p\ , &   } $$
where $i=(pr)^*$ and $h=(\delta H)^*$.
To compute such a term as an element of $\mathfrak{l}_{\BV}$, we should be able to evaluate it on any element $b$ of $\mathcal{H}$. For that, there is the following recursive procedure. Such an element, viewed as an element in $\Im\, Hd_\psi\subset G^{\ac}$, should be decomposed into two in all possible ways in the cooperad $G^{\ac}$. In the result $b_1\circ_i b_2$ of such a decomposition one should apply $i(\alpha)$ to one of the arguments, the remaining part of our left comb to the other argument, and anti-symmetrise with respect to that choice. 

Recall from \cite{DrummondColeVallette11} that $\Im\,  H d_\psi$ is weight graded by the number of vertices labelled by $\mu$; we denote it by  $(\Im\,  H d_\psi )^{[k]}$, for $k\ge 1$. This weight grading corresponds to the usual weight grading of the Koszul dual cooperad under the isomorphism 
$ \Im\,  H d_\psi\cong \HyperCom^{\ac}$. So a hypercommutative algebra structure is equivalent to a Maurer--Cartan element $\alpha\in \MC(\g_{\HyperCom})$ which vanishes outside $(\Im\,  H d_\psi )^{[1]}$. Tracing the above recipe for computing the left comb as an element of $\mathfrak{l}_{\BV}$, one sees that if $\alpha$ vanishes outside $(\Im\,  H d_\psi )^{[1]}$, then $\ell_n(r(z), \alpha, \ldots, \alpha)$ satisfies the same property. Indeed, suppose that we compute $\ell_n(r(z), \alpha, \ldots, \alpha)$ on an element $b$ from $(\Im\,  H d_\psi )^{[k]}$. Under the first decomposition we get one element to which we apply $i(\alpha)$ right away, and another element, to which we apply the rest of our left comb, which has $h$ at the root, amounting to applying $\delta H$. The former element must belong to $(\Im\,  H d_\psi )^{[1]}$ in order to be able to apply $i(\alpha)$ to it, hence the latter element will belong to $(\Im\,  H d_\psi )^{[k-1]}$, and after applying $\delta H$ will be back in $(\Im\,  H d_\psi 
)^{[k]}$. At the leaf level of the comb we shall end up with a zero contribution, since we shall be forced to apply $i(\alpha)$ to an element from $(\Im\,  H d_\psi )^{[k]}$ (and since $i(r)$ vanishes on $\Im\,  H d_\psi$). This completes the proof.
\end{proof}

\subsection{Proof of Theorem~\ref{thm:MainGiv=LinftyAction}}\label{subsec:EndProof}

\begin{proof}[Proof of Theorem~\ref{thm:MainGiv=LinftyAction}]
Let us examine the formulae in question carefully. Shuffle trees that we always use suggest that to make the notation most economic, we should replace each term $\ell_n(\alpha,\ldots,\alpha,r(z))$ in the formula for the infinitesimal gauge symmetry by the equal term $(-1)^{n-1}\ell_n(r(z),\alpha,\ldots,\alpha)$, and examine those terms without signs. At the stage when we deal with the Givental formulae, we shall also modify them accordingly. Let us note that the only non-trivial contributions to $\ell^\alpha_1(r(z))$ come from the terms $\ell_n(r_{n-2} z^{n-2}, \alpha, \ldots, \alpha)$. Indeed, if we evaluate the homotopy transfer formula for $\ell_n(\alpha, \ldots, \alpha, r(z))$ on a particular element $b$ of $\Im\,  H d_\psi$, we see that in the inductive computation of the result, as described in the proof of Lemma~\ref{lem:THCom},  the total power $\delta^{n-2}$ accumulates (each for one occurrence of~$h$). In the end, we apply $r(z)$ to that power of $\delta$, so only $r_{n-2}$ matters.  Moreover, the 
contribution of $\ell_n(r_{n-2} z^{n-2}, \alpha, \ldots, \alpha)$ to $\ell^\alpha_1(r(z))$ is precisely the left comb 
$$
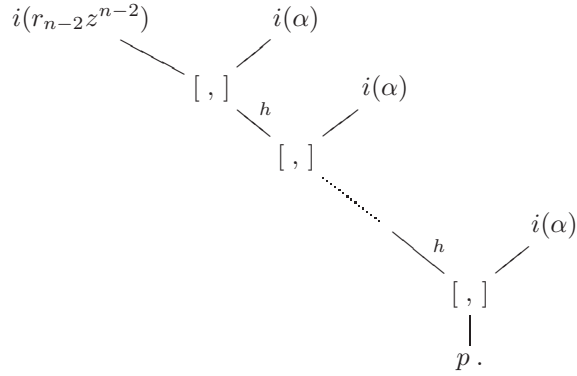
\xymatrix@=1em{i(r_{n-2}z^{n-2}) \ar@{-}[dr] & &i(\alpha)\ar@{-}[dl] & && \\
&\ar@{-}[dr]^h [ \; , \, ]&&i(\alpha)  \ar@{-}[dl]& &\\
&&\ar@{..}[dr][ \; , \, ]&& &\\
&&&\ar@{-}[dr]^h& &i(\alpha)  \ar@{-}[dl]\\
&&&&[ \; , \, ]\ar@{-}[d]& \\
&&&&p\ .&   } $$
The coefficient $\frac{1}{(n-1)!}$ disappears since there are $(n-1)!$ shuffle left combs of arity $n$.

Since both sides of the formula, we want to prove, are linear in $r(z)$, it is sufficient to prove the equality for each component $r_k$ individually. Furthermore, due to the factorisation property, a CohFT is completely defined by its values on fundamental cycles, and equivalently, a hypercommutative algebra is defined once we defined its generating operations. Since we know that both the Givental action and the gauge action take a hypercommutative algebra to a hypercommutative algebra, it is sufficient to show that the Givental formula, once integrated over fundamental cycles gives the same operations as the gauge symmetry formula on $(\Im\,  H d_\psi )^{[1]}$. We shall prove that by induction on $k$, showing that both satisfy the same kind of recursion relation. 

Let us denote by $\lambda^{(k)}_{n}\in\End_A(n)$ the value of the element $(-1)^{k-1}(\widehat{rz^k}.\alpha)_n$ on the fundamental cycle of $\overline{\mathcal{M}}_{0,n+1}$  (recall that we should change signs in the exact same way as we did for the gauge action), and by $\theta^{(k)}_n\in\End_A(n)$ the value of $\ell_{k+2}(rz^k,\alpha,\ldots,\alpha)$ on the $n$-ary generator from $(\Im\,  H d_\psi )^{[1]}$. 
Finally, let us denote by $\nu_n$ the value of the element $\alpha$ on the fundamental cycle of $\overline{\mathcal{M}}_{0,n+1}$. 
We remark that the elements $\lambda^{(k)}_{n}$ and $\theta^{(k)}_n$ are well defined for $k\ge 1$. 

For $k=0$, although the corresponding elements do not literally belong to the Givental formalism or to the gauge symmetries respectively, the actual formulae make sense and are applicable  as follows. The Givental action extends to the action of the Lie algebra $\End(A)[[z]]$ under the same Formula~(\ref{Givental-action-genus-0}). The gauge symmetry action can be generalised by considering the transfered homotopy cooperad structure on the augmented $\Sy$-module $\mathcal{H}\oplus \I$, where $I$ is concentrated in arity one $I=(0, \id, 0, \ldots)$, obtained from the extended cooperad structure on $\BV^{\ac}$ given by 
$$\mu \mapsto \Delta_{\overline{\BV}^{\ac}}(\mu)+
\sum_{m=1}^n \mu \circ_m \id +
\id \circ_1 \mu\ .$$  
It is straightforward to check that the image of ${\mathcal{H}}$ under the strictly higher structure maps of the 
homotopy 
 cooperad structure on $\cH \oplus \I$ produced by the homotopy transfer theorem for cooperads \cite[Theorem~$3.3$]{DrummondColeVallette11} remains the same. Only the cooperad part on $\cH$ of this homotopy cooperad structure on $\cH\oplus \I$ gets modified; it is given by the same extended formula as above for $\Delta_{\BV^{\ac}}$. 
 Writing $\cH \oplus \I={T}^c(\delta)   \oplus \Im\, Hd_\psi$, we are now working in the $L_\infty$-algebra 
$$\Hom({T}^c(\delta), \End(A)) \oplus \Hom_\Sy(\Im\, Hd_\psi, \End(A)) \cong \End(A)[[z]] \oplus \g_{\HyperCom} \ .$$
In the end, this generalisation modifies only $\theta^{(0)}$ by the preceeding argument.

Both of these two generalisations give the commutator in the endomorphism operad, for $k=0$:  
$$
\lambda^{(0)}_{n}=\theta^{(0)}_n=\bbb{-}r\circ_1\nu_{n}\bbb{+}\sum_{m=1}^n\nu_{n}\circ_m r\ .
 $$
We show in the remaining part of this section that for each $k\ge 0$:
\begin{gather*}
\lambda^{(k+1)}_{n}=\sum_{I\sqcup J=\underline{n}}\frac{\binom{|I|}{2}}{\binom{n}{2}}\lambda^{(k)}_{|J|+1}\circ_1\nu_{|I|}-\left(1-\frac{\binom{|I|}{2}}{\binom{n}{2}}\right)\nu_{|J|+1}\circ_1\lambda^{(k)}_{|I|},\\
\theta^{(k+1)}_{n}=\sum_{I\sqcup J=\underline{n}}\frac{\binom{|I|}{2}}{\binom{n}{2}}\theta^{(k)}_{|J|+1}\circ_1\nu_{|I|}-\left(1-\frac{\binom{|I|}{2}}{\binom{n}{2}}\right)\nu_{|J|+1}\circ_1\theta^{(k)}_{|I|}.
\end{gather*}
These formulae, together with the fact that $\lambda^{(0)}_n=\theta^{(0)}_n$ imply that $\lambda^{(k)}_n=\theta^{(k)}_n$ for all~$k$, which concludes the proof.
\end{proof}

\subsection{Recursion relation for the Givental action}

In this section, we shall prove the recursion relation for the Givental action stated above.

\begin{lemma}
The components of the Givental action on the fundamental classes satisfy the recurrence relation 
\begin{equation}\label{recursion-Givental}
\lambda^{(k+1)}_{n}=\sum_{I\sqcup J=\underline{n}}\frac{\binom{|I|}{2}}{\binom{n}{2}}\lambda^{(k)}_{|J|+1}\circ_1\nu_{|I|}-\left(1-\frac{\binom{|I|}{2}}{\binom{n}{2}}\right)\nu_{|J|+1}\circ_1\lambda^{(k)}_{|I|},\quad  \text{for} \quad k\ge 0\ .
\end{equation}
\end{lemma}

\begin{proof}
Proving \eqref{recursion-Givental} essentially amounts to somewhat imaginative application of Relations \eqref{TRR-0-root} and~\eqref{TRR-0-sym}. Let us explain how that is done. 

We evaluate $(-1)^k(\widehat{rz^{k+1}}.\alpha)_n$ using Formula \eqref{Givental-action-genus-0} as 
\begin{multline*}
r\circ_1\alpha_{n}\cdot\psi_0^{k+1}+(-1)^{k+2}\sum_{m=1}^n\alpha_{n}\cdot\psi_m^{k+1}\circ_m r+\\+
\sum_{\substack{I\sqcup J=\underline{n}, |I|\ge2,\\ i+j=k}}(-1)^{j+1}\tilde{\rho}_*\left(\big(\alpha_{|J|+1}\cdot\psi_1^j\big)\otimes \big(r \circ_1\alpha_{|I|}\cdot \psi_0^i\big)\right).
\end{multline*}

We shall represent this and subsequent formulae pictorially, so that 
 $$
r\circ_1\alpha_{n}\cdot\psi_0^{k+1}=
\vcenter{
\xymatrix@R=2mm@C=2mm{
\ar@{-}[dr]&\ldots&\ar@{-}[dl]\\
&*++[o][F-]{\underline{n}}\ar@{-}[d]&\\
&r\psi^{k+1}&\\
&&
}},
 $$
 $$
 \alpha_{n}\cdot\psi_m^{k+1}\circ_m r=
\vcenter{
\xymatrix@R=2mm@C=2mm{
\ldots\ar@{-}[dr]&r\psi^{k+1}\ar@{-}[d]&\ldots\ar@{-}[dl]\\
&*++[o][F-]{\underline{n}}\ar@{-}[d]&\\
&&
}},
 $$
and 
 $$
 \tilde{\rho}_*\left(\big(\alpha_{|J|+1}\cdot\psi_1^j\big)\otimes  \big(r \circ_1\alpha_{|I|}\cdot \psi_0^i\big)\right)=
\vcenter{
\xymatrix@R=.2mm@C=1.5mm{
 \ar@{-}[dr]&\cdots& \,\,\, \ar@{-}[dl] & \\
&*++[o][F-]{I}\ar@{-}[dr]& &&\\
&&\psi^i\ar@{-}[dr]&&\\
&&&r\ar@{-}[dr]& &\\
&&&&\psi^j\ar@{-}[dr]&\ldots&\, \ar@{-}[dl]&\\
&&&&&*++[o][F-]{J}\ar@{-}[d]&&&\\
&&&&& *{}
}}\hspace{-7mm},
 $$
so that the label of each vertex is the set of its ``free'' inputs. 

First, let us rewrite the first term in the formula for the Givental action using Formula \eqref{TRR-0-root}. That formula depends on a choice of $i_1,i_2\in\underline{n}$, and to obtain something symmetric, we shall average over all such choices.  {Using the factorisation property of genus 0 CohFT}, we obtain
\begin{equation}\label{pair-with-alpha-above-and-psi-below}
\vcenter{
\xymatrix@R=2mm@C=2mm{
\ar@{-}[dr]&\ldots&\ar@{-}[dl]\\
&*++[o][F-]{\underline{n}}\ar@{-}[d]&\\
&r\psi^{k+1}&\\
&&
}}
=
\sum_{I\sqcup J=\underline{n}}\frac{\binom{|I|}{2}}{\binom{n}{2}}
\vcenter{
\xymatrix@R=1mm@C=2mm{
\ar@{-}[dr]&\ldots& \,\,\, &\\
&*++[o][F-]{I}\ar@{-}[dr]\ar@{-}[ur]& \ldots& \ar@{-}[dl]&\\
&&*++[o][F-]{J}\ar@{-}[d]&&&\\
&&*{r\psi^{k}}&\\
&&
}}\hspace{-7mm}\ ,
\end{equation}
{where the tree on the right-hand side represents 
$   \tilde{\rho}_*\big((r\circ_1\psi_0^k.\alpha_{|J|+1})\otimes \alpha_{|I|}\big) $. }

To deal with the first sum in the formula, let us recall that, by Formula \eqref{TRR-0-sym}, we have
 $$
\vcenter{
\xymatrix@R=1mm@C=2mm{
\ldots\ar@{-}[dr]&r\psi^{k+1}\ar@{-}[d]&\ldots\ar@{-}[dl]\\
&*++[o][F-]{\underline{n}}\ar@{-}[d]&\\
&&
}}+ 
\vcenter{
\xymatrix@R=1mm@C=2mm{
\ldots\ar@{-}[dr]&r\psi^{k}\ar@{-}[d]&\ldots\ar@{-}[dl]\\
&*++[o][F-]{\underline{n}}\ar@{-}[d]&\\
&\psi\ar@{-}[d]&\\
&&
}} 
=
\sum_{I\sqcup J=\underline{n}}
\vcenter{
\xymatrix@R=1mm@C=2mm{
\ldots\ar@{-}[dr]&\ar@{-}[d]r\psi^k&\ldots\ar@{-}[dl]&\\
&*++[o][F-]{I}\ar@{-}[dr]& \ldots&\ar@{-}[dl]&\\
&&*++[o][F-]{J}\ar@{-}[d]&&&\\
&&
}}\hspace{-7mm}.
 $$
Since  Formula \eqref{TRR-0-root} gives, 
 $$
\vcenter{
\xymatrix@R=1mm@C=2mm{
\ldots\ar@{-}[dr]&r\psi^{k}\ar@{-}[d]&\ldots\ar@{-}[dl]\\
&*++[o][F-]{\underline{n}}\ar@{-}[d]&\\
&\psi\ar@{-}[d]&\\
&&
}} 
=
\sum_{I\sqcup J=\underline{n}}\frac{\binom{|I|}{2}}{\binom{n}{2}}
\vcenter{
\xymatrix@R=1mm@C=2mm{
\ldots\ar@{-}[dr]&\ar@{-}[d]r\psi^k&\ldots\ar@{-}[dl]&\\
&*++[o][F-]{I}\ar@{-}[dr]& \ldots&\ar@{-}[dl]&\\
&&*++[o][F-]{J}\ar@{-}[d]&&&\\
&&
}}\hspace{-7mm}+
\sum_{I\sqcup J=\underline{n}}\frac{\binom{|I|}{2}}{\binom{n}{2}}
\vcenter{
\xymatrix@R=1mm@C=2mm{
\ar@{-}[dr]&\ldots&\,\,\, &\\
&*++[o][F-]{I}\ar@{-}[dr]\ar@{-}[ur]& r\psi^k\ar@{-}[d]&\ldots\ar@{-}[dl]&\\
&&*++[o][F-]{J}\ar@{-}[d]&&&\\
&&
}}\hspace{-7mm},
 $$
then we have
\begin{multline}\label{pair-with-alpha-below-and-pair-with-alpha-above-and-psi-above}
(-1)^{k+2}\vcenter{
\xymatrix@R=1mm@C=2mm{
\ldots\ar@{-}[dr]&r\psi^{k+1}\ar@{-}[d]&\ldots\ar@{-}[dl]\\
&*++[o][F-]{\underline{n}}\ar@{-}[d]&\\
&&
}}= 
(-1)^{k+1}\sum_{I\sqcup J=\underline{n}}\left(\frac{\binom{|I|}{2}}{\binom{n}{2}}-1\right)
\vcenter{
\xymatrix@R=1mm@C=2mm{
\ldots\ar@{-}[dr]&\ar@{-}[d]r\psi^k&\ldots\ar@{-}[dl]&\\
&*++[o][F-]{I}\ar@{-}[dr]& \ldots&\ar@{-}[dl]&\\
&&*++[o][F-]{J}\ar@{-}[d]&&&\\
&&
}}\hspace{-7mm}+\\+(-1)^{k+1}
\sum_{I\sqcup J=\underline{n}}\frac{\binom{|I|}{2}}{\binom{n}{2}}
\vcenter{
\xymatrix@R=1mm@C=2mm{
\ar@{-}[dr]&\ldots&\ar@{-}[dl]\, \, &\\
&*++[o][F-]{I}\ar@{-}[dr]& r\psi^k\ar@{-}[d]&\ldots\ar@{-}[dl]&\\
&&*++[o][F-]{J}\ar@{-}[d]&&&\\
&&
}}\hspace{-7mm}.
\end{multline}

Let us rewrite the second sum in the formula. Let us first outline our strategy. For $i,j>0$, we  split the term
 $$
(-1)^{j+1}
\vcenter
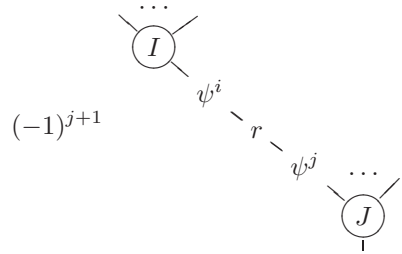
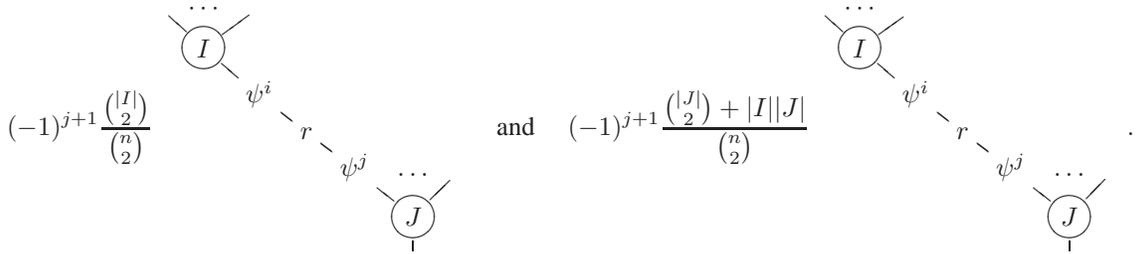
{
\xymatrix@R=.2mm@C=1.5mm{
\ar@{-}[dr]&\cdots&\,\, \ar@{-}[dl]&\\
&*++[o][F-]{I}\ar@{-}[dr]& &&\\
&&\psi^i\ar@{-}[dr]&&\\
&&&r\ar@{-}[dr]& &\\
&&&&\psi^j\ar@{-}[dr]&\ldots&\, \ar@{-}[dl]&\\
&&&&&*++[o][F-]{J}\ar@{-}[d]&&&\\
&&&&& *{}
}}\hspace{-7mm}
 $$
into two pieces, 
 $$
(-1)^{j+1}\frac{\binom{|I|}{2}}{\binom{n}{2}}
\vcenter{
\xymatrix@R=.2mm@C=1.5mm{
\ar@{-}[dr]&\cdots&\, \ar@{-}[dl]&\\
&*++[o][F-]{I}\ar@{-}[dr]& &&\\
&&\psi^i\ar@{-}[dr]&&\\
&&&r\ar@{-}[dr]& &\\
&&&&\psi^j\ar@{-}[dr]&\ldots&\ar@{-}[dl]\,\,&\\
&&&&&*++[o][F-]{J}\ar@{-}[d]&&&\\
&&&&&*{}
}}\hspace{-5mm}
\text{ and }\quad
(-1)^{j+1}\frac{\binom{|J|}{2}+|I||J|}{\binom{n}{2}}
\vcenter{
\xymatrix@R=.2mm@C=1.5mm{
\ar@{-}[dr]&\cdots&\ar@{-}[dl]\,\, &\\
&*++[o][F-]{I}\ar@{-}[dr]& &&\\
&&\psi^i\ar@{-}[dr]&&\\
&&&r\ar@{-}[dr]& &\\
&&&&\psi^j\ar@{-}[dr]&\ldots&\ar@{-}[dl]\, &\\
&&&&&*++[o][F-]{J}\ar@{-}[d]&&&\\
&&&&&*{}
}}\hspace{-7mm}.
 $$
We  rewrite the first term using Formula \eqref{TRR-0-root}, as in the first case above, and the second term using Formulae \eqref{TRR-0-sym} and \eqref{TRR-0-root}, as in the second case above. Then, we shall again average over various choices, although in the second case, the choices have to be somewhat more subtle than just all possibilities with equal coefficients. Finally, we  examine the contribution of extreme terms (where either $i$ or $j$ is equal to zero). We split them in a similar fashion into a weighted sum of two, but then perform similar computations with the two.

So rewriting 
 $$
(-1)^{j+1}\frac{\binom{|I|}{2}}{\binom{n}{2}}
\vcenter
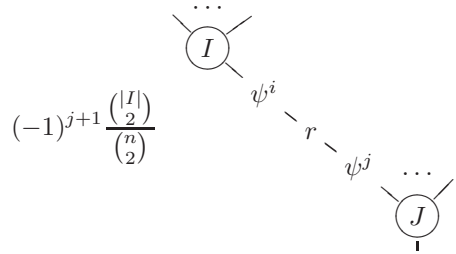
{
\xymatrix@R=.2mm@C=1.5mm{
\ar@{-}[dr]&\cdots&\ar@{-}[dl]\,\, &\\
&*++[o][F-]{I}\ar@{-}[dr]& &&\\
&&\psi^i\ar@{-}[dr]&&\\
&&&r\ar@{-}[dr]& &\\
&&&&\psi^j\ar@{-}[dr]&\ldots&\ar@{-}[dl]\, &\\
&&&&&*++[o][F-]{J}\ar@{-}[d]&&&\\
&&&&&*{}
}}\hspace{-7mm}
 $$
using Formula \eqref{TRR-0-root} yields
\begin{equation}\label{triple-with-alpha-above}
(-1)^{j+1}\sum_{I_1\sqcup I_2\sqcup J=\underline{n}}\frac{\binom{|I_1|}{2}}{\binom{n}{2}}
\vcenter{
\xymatrix@R=.2mm@C=1.5mm{
\ar@{-}[dr]&\cdots&\ar@{-}[dl]\,  &\\
&*++[o][F-]{I_1}\ar@{-}[dr]&\ \ldots&\ar@{-}[dl]\, &\\
&&*++[o][F-]{I_2}\ar@{-}[dr]& &&\\
&&&\psi^{i-1}\ar@{-}[dr]&&\\
&&&&r\ar@{-}[dr]& &\\
&&&&&\psi^j\ar@{-}[dr]&\ldots&\ar@{-}[dl]\,&\\
&&&&&&*++[o][F-]{J}\ar@{-}[d]&&&\\
&&&&&&*{}&&&&
}}\hspace{-7mm},
\end{equation}
Let us keep the factor $1-\frac{\binom{|I|}{2}}{\binom{n}{2}}=\frac{\binom{|J|}{2}+|I||J|}{\binom{n}{2}}$ aside for the moment, and rewrite 
 $$
(-1)^{j+1}
\vcenter
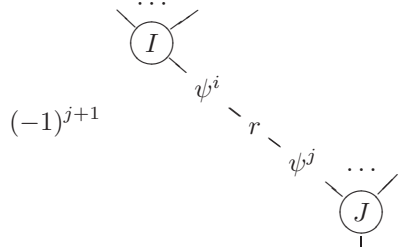
{
\xymatrix@R=.2mm@C=1.5mm{
\ar@{-}[dr]&\cdots&\ar@{-}[dl]\, &\\
&*++[o][F-]{I}\ar@{-}[dr]& &&\\
&&\psi^i\ar@{-}[dr]&&\\
&&&r\ar@{-}[dr]& &\\
&&&&\psi^j\ar@{-}[dr]&\ldots&\ar@{-}[dl]\,&\\
&&&&&*++[o][F-]{J}\ar@{-}[d]&&&\\
&&&&&*{}
}}\hspace{-7mm}
 $$
using Formula \eqref{TRR-0-sym}, which yields
 $$
(-1)^{j+1}
\sum_{I\sqcup J_1\sqcup J_2=\underline{n}}
\vcenter
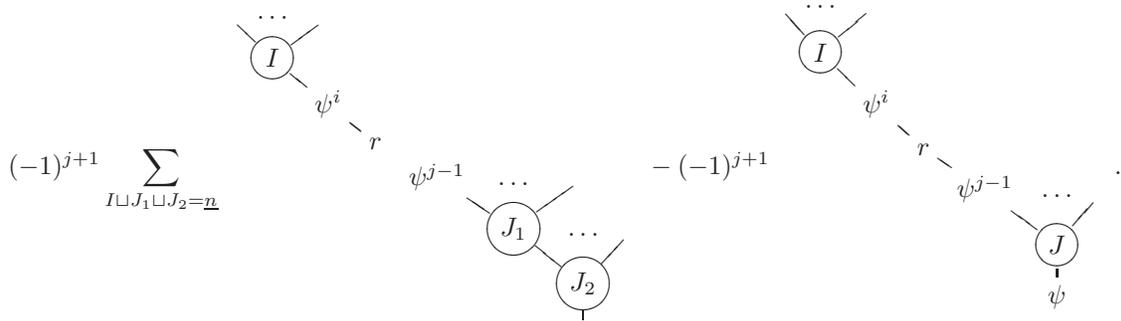
{
\xymatrix@R=.2mm@C=1.5mm{
\ar@{-}[dr]&\cdots&\ar@{-}[dl]\, &\\
&*++[o][F-]{I}\ar@{-}[dr]& &&\\
&&\psi^i\ar@{-}[dr]&&\\
&&&r\ar@{-}[dr]& &\\
&&&&\psi^{j-1}\ar@{-}[dr]&\ldots&\ar@{-}[dl]\, &\\
&&&&&*++[o][F-]{J_1}\ar@{-}[dr]&\ldots&\ar@{-}[dl]\, &\\
&&&&&&*++[o][F-]{J_2}\ar@{-}[d]&&&\\
&&&&&&*{}&&
}}\hspace{-7mm}-
(-1)^{j+1}
\vcenter{
\xymatrix@R=1mm@C=1.5mm{
\ar@{-}[dr]&\cdots&\ar@{-}[dl]\, &\\
&*++[o][F-]{I}\ar@{-}[dr]& &&\\
&&\psi^i\ar@{-}[dr]&&\\
&&&r\ar@{-}[dr]& &\\
&&&&\psi^{j-1}\ar@{-}[dr]&\ldots&\ar@{-}[dl]\, &\\
&&&&&*++[o][F-]{J}\ar@{-}[d]&&&\\
&&&&&\psi&\\
&&&&&&
}}\hspace{-7mm}.
 $$
Recalling the constant factors, we shall keep
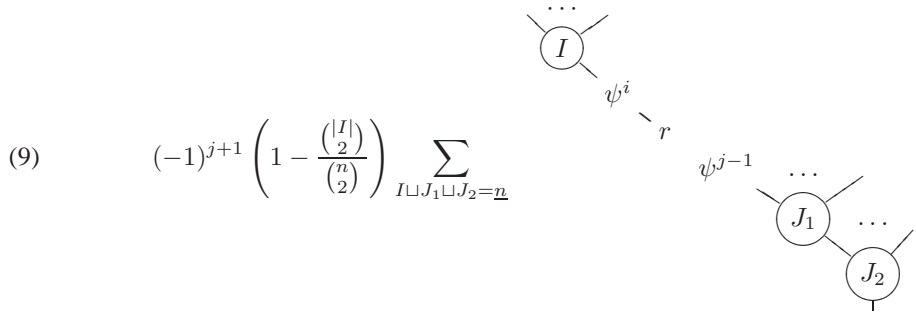
\begin{equation}\label{part1adding-up}
(-1)^{j+1}
\left(1-\frac{\binom{|I|}{2}}{\binom{n}{2}}\right)
\sum_{I\sqcup J_1\sqcup J_2=\underline{n}}
\vcenter{
\xymatrix@R=.2mm@C=1.5mm{
\ar@{-}[dr]&\cdots&\ar@{-}[dl]\, &\\
&*++[o][F-]{I}\ar@{-}[dr]& &&\\
&&\psi^i\ar@{-}[dr]&&\\
&&&r\ar@{-}[dr]& &\\
&&&&\psi^{j-1}\ar@{-}[dr]&\ldots&\, \ar@{-}[dl]&\\
&&&&&*++[o][F-]{J_1}\ar@{-}[dr]&\ldots&\ar@{-}[dl]\, &\\
&&&&&&*++[o][F-]{J_2}\ar@{-}[d]&&&\\
&&&&&&*{}&&
}}
\end{equation}
as is, and rewrite
 $$
(-1)^{j}
\frac{\binom{|J|}{2}+|I||J|}{\binom{n}{2}}
\vcenter
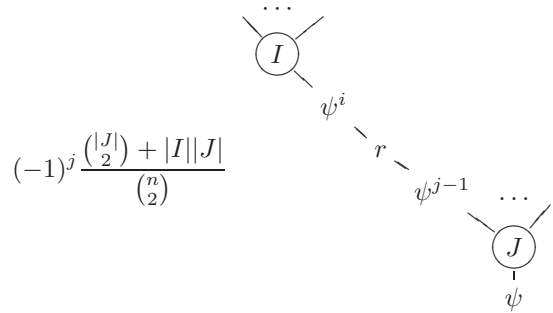
{
\xymatrix@R=1mm@C=1.5mm{
\ar@{-}[dr]&\cdots&\, \ar@{-}[dl]&\\
&*++[o][F-]{I}\ar@{-}[dr]& &&\\
&&\psi^i\ar@{-}[dr]&&\\
&&&r\ar@{-}[dr]& &\\
&&&&\psi^{j-1}\ar@{-}[dr]&\ldots&\, \ar@{-}[dl]&\\
&&&&&*++[o][F-]{J}\ar@{-}[d]&&&\\
&&&&&\psi&\\
&&&&&&
}}\hspace{-7mm}
 $$
using Formula \eqref{TRR-0-root}. We shall average it over the variety of $\binom{|J|}{2}+|I||J|$ different choices of two leaves: $\binom{|J|}{2}$ choices where both leaves belong to $J$, and ``$|I||J|$ choices where one leaf belongs to $I$ and the other leaf belongs to $J$'', or more precisely the choices where one leaf belongs to $J$, and the other leaf is the ``connector'' between the two corollas, taken with multiplicity $|I|$. The result is made up of 
\begin{equation}\label{part2adding-up}
(-1)^{j}
\frac{|I||J_1|}{\binom{n}{2}}
\sum_{I\sqcup J_1\sqcup J_2=\underline{n}}
\vcenter{
\xymatrix@R=1mm@C=1.5mm{
\ar@{-}[dr]&\cdots&\, \ar@{-}[dl]&\\
&*++[o][F-]{I}\ar@{-}[dr]& &&\\
&&\psi^i\ar@{-}[dr]&&\\
&&&r\ar@{-}[dr]& &\\
&&&&\psi^{j-1}\ar@{-}[dr]&\ldots&\, \ar@{-}[dl]&\\
&&&&&*++[o][F-]{J_1}\ar@{-}[dr]&\ldots&\, \ar@{-}[dl]&\\
&&&&&&*++[o][F-]{J_2}\ar@{-}[d]&&&\\
&&&&&&*{}&&
}}\hspace{-7mm},
\end{equation}
\begin{equation}\label{part3adding-up}
(-1)^{j}
\frac{\binom{|J_1|}{2}}{\binom{n}{2}}
\sum_{I\sqcup J_1\sqcup J_2=\underline{n}}
\vcenter{
\xymatrix@R=.2mm@C=1.5mm{
\ar@{-}[dr]&\cdots&\, \ar@{-}[dl]&\\
&*++[o][F-]{I}\ar@{-}[dr]& &&\\
&&\psi^i\ar@{-}[dr]&&\\
&&&r\ar@{-}[dr]& &\\
&&&&\psi^{j-1}\ar@{-}[dr]&\ldots&\, \ar@{-}[dl]&\\
&&&&&*++[o][F-]{J_1}\ar@{-}[dr]&\ldots&\, \ar@{-}[dl]&\\
&&&&&&*++[o][F-]{J_2}\ar@{-}[d]&&&\\
&&&&&&*{}&&
}}\hspace{-7mm},
\end{equation}
and
\begin{equation}\label{fork-with-alpha-above}
(-1)^{j}
\frac{\binom{|J_1|}{2}}{\binom{n}{2}}
\sum_{I\sqcup J_1\sqcup J_2=\underline{n}}
\vcenter{
\xymatrix@R=1mm@C=1.5mm{
\ar@{-}[dr]&\cdots&\, \ar@{-}[dl]&\\
&*++[o][F-]{I}\ar@{-}[dr]& &&\\
&&\psi^i\ar@{-}[dr]&&\\
&&&r\ar@{-}[dr]& &\, \ar@{-}[dr]&\ldots& \ar@{-}[dl]\\
&&&&\psi^{j-1}\ar@{-}[dr]&\ldots&*++[o][F-]{J_1}\ar@{-}[dl]&\\
&&&&&*++[o][F-]{J_2}\ar@{-}[d]&&\\
&&&&&*{}&&&.
}}\hspace{-7mm}
\end{equation}
Let us remark that the terms \eqref{part1adding-up}, \eqref{part2adding-up}, and \eqref{part3adding-up} collect altogether into
\begin{equation}\label{triple-with-alpha-below}
(-1)^{j}
\left(\frac{\binom{|I\sqcup J_1|}{2}}{\binom{n}{2}}-1\right)
\sum_{I\sqcup J_1\sqcup J_2=\underline{n}}
\vcenter{
\xymatrix@R=1mm@C=1.5mm{
\ar@{-}[dr]&\cdots&\, \ar@{-}[dl]&\\
&*++[o][F-]{I}\ar@{-}[dr]& &&\\
&&\psi^i\ar@{-}[dr]&&\\
&&&r\ar@{-}[dr]& &\\
&&&&\psi^{j-1}\ar@{-}[dr]&\ldots&\, \ar@{-}[dl]&\\
&&&&&*++[o][F-]{J_1}\ar@{-}[dr]&\ldots&\ar@{-}[dl]&\\
&&&&&&*++[o][F-]{J_2}\ar@{-}[d]&&&\\
&&&&&&*{}&&
}}\hspace{-7mm},
\end{equation}
There are just two terms where some of our manipulations would not work (they correspond to $i=0$ and to $j=0$, where either the top $\psi$-class or the bottom $\psi$-class is missing, and therefore only one of the two weighted parts into which we split the respective term would be rewritten). Those additional contributions are
\begin{equation}\label{pair-with-alpha-above-and-psi-above-missing}
(-1)^{k+1}
\frac{\binom{|I|}{2}}{\binom{n}{2}}
\vcenter{
\xymatrix@R=.2mm@C=1.5mm{
\ar@{-}[dr]&\cdots&\, \ar@{-}[dl]&\\
&*++[o][F-]{I}\ar@{-}[dr]& &&\\
&&r\ar@{-}[dr]& &\\
&&&\psi^k\ar@{-}[dr]&\ldots&\ar@{-}[dl]&\\
&&&&*++[o][F-]{J}\ar@{-}[d]&&&\\
&&&&*{}
}}\hspace{-7mm}
\end{equation}
and
\begin{equation}\label{pair-with-alpha-below-and-psi-below-missing}
-
(1-\frac{\binom{|I|}{2}}{\binom{n}{2}})
\vcenter{
\xymatrix@R=.2mm@C=1.5mm{
\ar@{-}[dr]&\cdots&\, \ar@{-}[dl]&\\
&*++[o][F-]{I}\ar@{-}[dr]& &&\\
&&\psi^{k}\ar@{-}[dr]&&\\
&&&r\ar@{-}[dr]& \ldots&\ar@{-}[dl]&\\
&&&&*++[o][F-]{J}\ar@{-}[d]&&&\\
&&&&*{}
}}\hspace{-7mm}.
\end{equation}
Finally, {we integrate all the above terms over  $\overline{\mathcal{M}}_{0,|J|+2}\times \overline{\mathcal{M}}_{0,|I|+1}$.}
We notice that the term $$\frac{\binom{|I|}{2}}{\binom{n}{2}}\lambda^{(k)}_{|J|+1}\circ_1\nu_{|I|}$$ is assembled precisely out of the contributions
of Formula~\eqref{pair-with-alpha-above-and-psi-below}, the second half of Formula~\eqref{pair-with-alpha-below-and-pair-with-alpha-above-and-psi-above},
and Formulae~\eqref{triple-with-alpha-above}, \eqref{fork-with-alpha-above}, and~\eqref{pair-with-alpha-above-and-psi-above-missing}, while the term 
$$-\left(1-\frac{\binom{|I|}{2}}{\binom{n}{2}}\right)\nu_{|J|+1}\circ_1\lambda^{(k)}_{|I|}$$ is assembled precisely out of the contributions
of the first half of Formula~\eqref{pair-with-alpha-below-and-pair-with-alpha-above-and-psi-above}, and Formulae \eqref{triple-with-alpha-below} and~\eqref{pair-with-alpha-below-and-psi-below-missing}.
\end{proof}

\subsection{Recursion relation for the  gauge symmetries action}

Let us now prove the same recursion relation for the  gauge symmetries action.

\begin{lemma}
The components of the  gauge symmetries action on the generators of $\HyperCom$ satisfy the recurrence relation
\begin{equation}\label{recursion-gauge}
\theta^{(k+1)}_{n}=\sum_{I\sqcup J=\underline{n}}\frac{\binom{|I|}{2}}{\binom{n}{2}}\theta^{(k)}_{|J|+1}\circ_1\nu_{|I|}-\left(1-\frac{\binom{|I|}{2}}{\binom{n}{2}}\right)\nu_{|J|+1}\circ_1\theta^{(k)}_{|I|},\quad  \text{for} \quad k\ge 0\ .
\end{equation}
\end{lemma}
\begin{proof}
Recall from \ref{subsec:EndProof} that the left-hand side of this formula  is given by the term $\ell_{k+3}(r z^{k+1}, \alpha, \ldots, \alpha)$ inside $\ell^\alpha_1(r(z))$, which reduces to the left comb 
$$C:=\ \ \vcenter{\xymatrix@=1em{i(r z^{k+1}) \ar@{-}[dr] & &i(\alpha)\ar@{-}[dl] & && \\
&\ar@{-}[dr]^h [ \; , \, ]&&i(\alpha)  \ar@{-}[dl]& &\\
&&\ar@{..}[dr][ \; , \, ]&& &\\
&&&\ar@{-}[dr]^h& &i(\alpha)  \ar@{-}[dl]\\
&&&&[ \; , \, ]\ar@{-}[d]& \\
&&&&p&   }} $$
with the $k+1$ edges labelled by~$h$. Let us start evaluating this map on the element $b$ of 
$(\Im\,  H d_\psi )^{[1]}(n)$ representing the $n$-ary generator of~$\HyperCom$. 
It follows from \cite{DrummondColeVallette11} that $b=H\left(\sum t\right)$, where $t$ ranges over all shuffle binary trees with $n$ leaves and with internal vertices  labelled by $\beta$. Thus, the element $b$ is equal to the weighted sum of shuffle binary trees with one internal vertex labelled by $\mu$ and the others labelled by $\beta$. The application of the left comb $C$ to $b$ amounts to computing the weighted sum of elements 
 $$
C'(\delta H(b'))\circ_i i(\alpha)(b'')-i(\alpha)(b')\circ_i C'(\delta H(b''))\ ,
 $$
over all ways to split $b$ as a decomposition $b'\circ_i b''$, where $b'$ has the $J\sqcup\{i\}$ as its set of leaves, $b''$ has $I$ as its set of leaves, and where $C'$ denotes the ``top part'' of the left comb. Since $\alpha$ vanishes outside $(\Im\,  H d_\psi )^{[1]}$, we may only apply $i(\alpha)$ to the part of the decomposition that contains the only corolla labelled by $\mu$. This means that in the term $C'(\delta H(b'))\circ_i i(\alpha)(b'')$ all weights coming from the formula for $H$ are within $b''$. It follows that $b'$ is the sum of all shuffle binary trees with the set of leaves $J\sqcup\{i\}$, and $C'(\delta H(b'))$ is nothing but $\theta^{(k)}_{|J|+1}$, since the additional occurrence of $\delta$ we now have will only affect the power $z$ at the last stage. 
To write  $b''$ as an element of the  image of $H$ we just need to modify the denominators of the weights; all numerators are automatically correct:
$$i(\alpha)(b'')=\frac{\binom{|I|}{2}}{\binom{n}{2}}\alpha(Hd_\psi(b''))\ . $$
Similarly, in the term $i(\alpha)(b')\circ_i C'(\delta H(b''))$, all weights coming from the formula for $H$ are within $b'$, so $b''$ is the sum of all shuffle binary trees with the set of leaves $I$, and $C'(\delta H(b''))$ is nothing but $\theta^{(k)}_{|I|}$. In $b'$ the weights are not quite correct, however, when applying $i(\alpha)$ to $b'$, one uses the projection $H d_\psi$, which turns out to create the correct weights. Indeed, since each term in $b'$ is a tree monomial containing exactly one vertex labelled $\mu$, its image under $d_\psi$ is is a tree monomial of the same shape where all vertices are labelled $\beta$, and then the application of $H$ to that monomial creates correct weights. Thus $i(\alpha)(b')$ differs from $\nu_{|J|+1}$ by a scalar multiple, which is the sum of the $H$-weights of the vertices of~$b'$ computed for those vertices viewed as vertices of~$b$. For each tree monomial that sum is equal to $1-\frac{\binom{|I|}{2}}{\binom{n}{2}}$ since  the total sum of weights of all vertices of a given tree is equal to~$1$ and since we already noticed that for $b''$ the sum of weights is equal to~$\frac{\binom{|I|}{2}}{\binom{n}{2}}$. This completes the  proof. 
\end{proof}

\subsection{Givental action on homotopy hypercommutative algebras}\label{subsec:GiventalActionOnHomotopyHycom}

If in Theorem~\ref{thm:MainGiv=LinftyAction} we assume $A$ to be a chain complex with a non-zero differential, the result remains true for the modified statement 
$$\ell^\alpha_1(r(z))=[d_A,r(z)]+\widehat{r(z)}.\alpha\ . $$
In particular, the Givental formulae define an action if we restrict ourselves to the subalgebra of the Givental Lie algebra consisting of elements that commute with $d_A$. 
(This restriction makes sense, as only this way the Givental formulae are homotopically meaningful). A direct consequence of this theorem is that the Givental infinitesimal action on hypercommutative algebras extends naturally to  homotopy hypercommutative algebras. This suggests the following definition. 

\begin{definition}\label{def:DefinitionForHomotopyHypercom}
Let a homotopy hypercommutative algebra structure on a chain complex $A$ be encoded by a Maurer--Cartan element $\alpha\in\g_{\HyperCom}$, and let $r(z)$ be a degree $0$ element of~$\g_{\Delta}$ commuting with $d_A$. The \emph{higher infinitesimal Givental action} is the gauge symmetry action of $r(z)$ on $\alpha$:
$$\widehat{r(z)}.\alpha:=\ell^\alpha_1(r(z))\ . $$
\end{definition}

\section{Givental action as an action of trivialisations of the circle action}\label{sec:Givental=Triv}

In the previous sections, we have been able to describe the  infinitesimal Givental action as an infinitesimal gauge symmetry in the framework of homotopy Batalin--Vilkovisky algebras. We now integrate this infinitesimal action and interpret Givental group action as an action of trivialisations of the trivial circle action. 

\subsection{Homotopy BV-algebras with trivialisation of the circle action}
Up to homotopy, the action of the circle is modelled by the Koszul resolution $\Omega \, H^\bullet(S^1)^{\ac} = \big(T\left(s^{-1} \bar{T}^c(\delta)  \right), d\big)$ of the Koszul algebra $H^\bullet(S^1)$. Recall that modules over the dg algebra  $\Omega \, H^\bullet(S^1)^{\ac}$ are called multicomplexes.
There is a notion of morphism of multicomplexes which encodes their homotopy properties; these morphisms are called \emph{$\infty$-morphisms} and are actually made up of  collections of maps, see \cite{DotsenkoShadrinVallette12}. 
An important class of invertible $\infty$-morphisms is \emph{$\infty$-isotopies}; these are $\infty$-morphisms whose first map is the identity. 
A (homotopy) trivialisation of a circle action amounts to an $\infty$-isotopy from this structure to the trivial circle action, see \cite{DotsenkoShadrinVallette14}.

The dg Lie algebra $\g_\Delta=\Hom\left(\bar{T}^c(\delta), \End(A)\right)$, which models multicomplex structures on $A$, embeds naturally into the unital dg associative algebra 
$$\a_\Delta:=\left(\Hom\left({T}^c(\delta), \End(A)\right), \star\right)\ . $$ 
In this convolution algebra, we denote by $1$ and by $\partial$ the maps defined respectively by 
$$ 1 \ : \ 1 \mapsto \id_A,  \quad     \partial \ : \ 1 \mapsto \partial_A,\quad  \text{and}\quad \delta^k\mapsto 0, \ \ \text{for}\ \ k\ge 1\ . $$ 
(Since the convolution algebra  $\a_\Delta$ is weight graded connected, it admits exponential and logarithm maps, see \cite[Section~$3$]{DotsenkoShadrinVallette14} for more details.)
The data of a trivialisation of a circle action $\phi \in \MC(\g_\Delta)$ on $A$ amounts to a degree $0$ element 
$f\in \g_\Delta$ satisfying the following equation in the algebra $\mathfrak{a}_\Delta$:
$$(1+f)\star(\phi +\partial)=\partial \star (1+f)\ . $$
This data is equivalent to  a module structure over the quasi-free dg algebra 
$$\left(T\left(s^{-1} \bar{T}^c(\delta) \oplus  \bar{T}^c(\delta) \right), d_1+d'_2+d_2'' \right)\ ,$$ 
where the differential $d_1$ is the unique derivation which extends the desuspension map on the space of generators 
$s^{-1}\colon  \bar{T}^c(\delta) \to s^{-1}\bar{T}^c(\delta)$, where 
the differential $d'_2$ is the unique derivation which extends the coproduct map followed by the desuspension map  
$$ \bar{T}^c(\delta) \to \bar{T}^c(\delta)\otimes \bar{T}^c(\delta)
\xrightarrow{s^{-1}}  s^{-1}\bar{T}^c(\delta)\otimes \bar{T}^c(\delta)\ ,$$
 and where the differential $d_2''$ is the differential coming form the Koszul resolution
$\Omega \, H^\bullet(S^1)^{\ac}$.\\
 
Let us enlarge the picture one step further and consider now the data of homotopy Batalin--Vilkovisky algebras together with a (homotopy) trivialisation of the circle action. This data  is  encoded by the following quasi-free operad
$$\trBV_\infty:=
\left( \TTT\big( H^{\bullet}(\mathcal{M}_{0,n+1})\oplus s^{-1}\bar{T}^c(\delta) \oplus \bar{T}^c(\delta)\big), 
d:=d_1+d_2+d_3+\cdots \right)\ ,
$$
where $d_1$ is the same kind of derivation as above, where $d_2$ is the sum of the derivations coming from $d_2'$, $d_2''$ and the quadratic part of the differentials of $\sBV_\infty$, and where $d_3, d_4, \ldots$ are the higher components of the differential of the operad $\sBV_\infty$.
Notice that the  operad $\trBV_\infty$ is actually the coproduct of the operad $\sBV_\infty$ with the algebra 
  $\big(T\left(s^{-1} \bar{T}^c(\delta) \oplus  \bar{T}^c(\delta) \right), d_1+d'_2+d_2'' \big)$ over the algebra $\Omega \, H^\bullet(S^1)^{\ac}$ .

Algebra structures over the operad $\trBV_\infty$ are encoded by the Maurer--Cartan elements of the convolution $L_\infty$-algebra 
$$\mathfrak{l}_{\trBV}:=\big(
\Hom_\Sy(H^{\bullet+1}(\mathcal{M}_{0,n+1})\oplus \bar{T}^c(\delta) \oplus s\, \bar{T}^c(\delta), \End_A),  (\partial_A)_*+(d_1)^*, \ell_2,  \ell_3, \ldots
\big)\ .$$
The underlying vector space of this $L_\infty$-algebra is the direct sum 
$$\mathfrak{l}_{\trBV} \cong  \g_{\HyperCom}\oplus \g_{\Delta}  \oplus s^{-1}\g_{\Delta}\   $$
of its Lie sub-algebras, where the Lie algebra structure on $s^{-1}\g_{\Delta}$ is abelian; 
as at the end of  Section~\ref{subsec:Linfty}, the $L_\infty$-algebra  structure is an extension in the category of $L_\infty$-algebras.

So its Maurer--Cartan elements  are sums of three terms $\alpha+\phi+\rho$, corresponding respectively to the (homotopy) hypercommutative part, the (homotopy) circle action and its (homotopy) trivialisation. Let us consider the infinitesimal gauge symmetry $\ell_1^{\alpha+\phi+\rho}(\lambda)$ associated to a degree $0$ element $\lambda$ of $\mathfrak{l}_{\trBV}$.

\begin{lemma}\label{lem:CVt=1}
For any degree $0$ element $\lambda$ of $\g_{\Delta}\subset \mathfrak{l}_{\trBV}$, the associated infinitesimal gauge symmetry integrates to time $t=1$, i.e.  the differential equation 
\begin{eqnarray*}
\gamma'(t)=\ell_1^{\gamma(t)}(\lambda), \qquad \gamma(0)=\alpha+\phi+\rho
\end{eqnarray*}
admits a solution which converges at time $t=1$. 
\end{lemma}

\begin{proof}
An element $\lambda \in  \mathfrak{l}_{\trBV}$ lies in $\g_{\Delta}$ if and only if it vanishes on 
$H^{\bullet+1}(\mathcal{M}_{0,n+1})\oplus  s\, \bar{T}^c(\delta)$.
Let us decompose $\gamma(t)$ into $\alpha(t)+\phi(t)+\rho(t)$. The differential equation 
\begin{eqnarray*}
\gamma'(t)=\alpha'(t)+\phi'(t)+\rho'(t)&=& \ell_1^{\gamma(t)}(\lambda)
=
\sum_{n\ge1}\frac{1}{(n-1)!}\, \ell_n\big(\gamma(t), \ldots, \gamma(t), \lambda\big)\\
&=&
\sum_{n\ge1}\frac{1}{(n-1)!}\, \ell_n\big(\alpha(t)+\phi(t)+\rho(t), \ldots, \alpha(t)+\phi(t)+\rho(t), \lambda\big) \ ,
\end{eqnarray*}
decomposes as follows. 
Since $\bar{T}^c(\delta)$ is a sub-cooperad of the homotopy cooperad  on the generators of the quasi-free operad $\trBV_\infty$, 
we conclude that 
 $\phi(t)$ satisfies the classical differential equation \cite{GoldmanMillson88}
$$\phi'(t)=\ell_1(\lambda)+\ell_2(\phi(t), \lambda)=\partial_A \lambda + [\phi(t), \lambda   ]\ , \qquad \phi(0)=\phi $$ 
in $\g_\Delta$. Therefore  $\phi(t)$ is equal to 
$$\phi(t)=e^{-t\ad_{\lambda}}(\phi) -\frac{e^{-t\ad_{\lambda}}-\id}{t\ad_{\lambda}}(\partial_A \lambda)\ .$$
Recall that the dg Lie algebra $\g_\Delta$ is weight graded and that the elements $\phi$ and $\lambda$ are concentrated in positive weight, so, applied to an element $\delta^k$ of $\bar{T}^c(\delta)$, the element $\phi(t)$ becomes a  finite sum. Let us denote by 
$\phi_k(t):=\phi(t)(\delta^k)$, the component of $\phi(t)$ on $\delta^k$. One can see  that 
$\phi_k(t)$ is a polynomial in $t$ of degree at most $k$. This shows that $\phi(t)$ is well-defined in $\g_\Delta$ for any $t$, in particular  at $t=1$. 

The map $\chi$ provides $s^{-1}\g_{\Delta}$ with a left module structure, denoted $\bar\star$,  of the convolution unital associative algebra
$\a_\Delta$. With respect to this structure, the differential equation satisfied by $\rho(t)$ is 
$$\rho'(t)=\lambda \, \bar\star\,  \rho(t) + s^{-1} \lambda\ ,  \qquad \rho(0)=\rho\ .$$
Therefore, $\rho(t)$ is equal to 
$$\rho(t)=s^{-1}\left(  
e^{t\lambda}\star (1+s\, \rho) -1
\right)\ ,$$
where the  exponential  $e^{t\lambda}=1 + t\lambda+\frac12 t^2\lambda \star \lambda +\cdots$ lives in the algebra $\a_\Delta$. By the same weight grading argument as above, $\rho(t)$ is well-defined for any $t$, in particular at $t=1$.  

The map $\alpha(t)$ is the component of $\gamma(t)$ on $H^{\bullet+1}(\mathcal{M}_{0,n+1})=\Im\, H d_\psi$. 
Let us consider the grading on $\Im\, H d_\psi$ defined by the number of vertices labelled by $\beta$. We denote it by 
${\Im\, H d_\psi}^{<k>}$ and we denote by $\alpha_k(t)$ the restriction of $\alpha(t)$ on it. 
Finally, we consider the weight grading defined on 
$$\mathcal{H}= H^{\bullet+1}(\mathcal{M}_{0,n+1})\oplus  \bar{T}^c(\delta) $$
by  this grading on the left-hand side and by the $\delta$-grading on the right-hand side. The formula for the homotopy transfer theorem of homotopy cooperads \cite[Theorem~$3.3$]{DrummondColeVallette11} shows that the homotopy cooperad structure on $\mathcal{H}$ preserves this weight grading. 

We are now ready to prove, by induction on $k$, that $\alpha_k(t)$ is a polynomial in $t$ of degree at most $k$. 
First notice that $\alpha(t)$ satisfies the differential equation 
\begin{eqnarray*}
\alpha'(t)
=
\frac12  \ell_3\big(\alpha(t)+\phi(t), \alpha(t)+\phi(t), \lambda\big) + 
\frac16  \ell_4\big(\alpha(t)+\phi(t), \alpha(t)+\phi(t), \alpha(t)+\phi(t), \lambda\big) + \cdots \ ,
\end{eqnarray*}
by the shape of the homotopy cooperad structure on $\mathcal{H}$. For instance, the  term $\ell_2$ of the differential equation comes from the cooperad structure on $H^{\bullet+1}(\mathcal{M}_{0,n+1})$; but this cooperad structure produces no $\delta$ term, so $\ell_2(\alpha(t)+\phi(t), \lambda)\allowbreak=0$ because  the map $\lambda$ vanishes on $H^{\bullet+1}(\mathcal{M}_{0,n+1})$. 
So, for $k=0$, we have 
$\alpha_0'(t)=0$, which shows that $\alpha_0(t)=\alpha_0$ is constant. 
Using the above weight grading of the homotopy cooperad $\mathcal{H}$, one can see 
that the image of ${\Im\, H d_\psi}^{<k>}$ under the homotopy cooperad structure map $\Delta_m$ lives in  
trees with $m$ vertices labelled respectively by 
$$ 
{\Im\, H d_\psi}^{<k_1>}, \ \ldots,\  {\Im\, H d_\psi}^{<k_i>}, \ \KK\, \delta^{l_1},\  \ldots, \ \KK\,  \delta^{l_j}\ , 
$$
with $i+j=m$ and $k_1+\cdots+k_i+l_1+\cdots+l_j=k$. 
Therefore, $\alpha'_k(t)$ is equal to a linear combination of terms of the form 
$$\ell_m\big(
\alpha_{k_1}(t), \ldots, \alpha_{k_i}(t), \phi_{l_1}(t), \ldots, \phi_{l_{j-1}}(t), \lambda_{l_j}
\big) \ , $$
which are polynomial in $t$ of degree at most $k_1+\cdots+k_i+l_1+\cdots+l_{j-1}<k$, by the induction hypothesis. 
This proves that $\alpha_k(t)$ is polynomial in $t$ of degree at most $k$. So, $\alpha(t)$ exists for any $t$, in particular $t=1$, which concludes the proof. 
\end{proof}

\subsection{Givental morphism}
Given a Maurer--Cartan element $\alpha+\phi+\rho$ in $\mathfrak{l}_{\trBV}$, we consider the degree $0$ element 
$$\lambda:=-\ln (1+s\, \rho)$$
defined in $\a_\Delta$, where $s\, \rho$ is the composite 
$$\bar{T}^c(\delta) \stackrel{s}{\longrightarrow} s\, \bar{T}^c(\delta) \stackrel{\rho}{\longrightarrow} \End(A)\ .$$
The element $\lambda$ actually lives inside $\g_\Delta$, and we extend it trivially to an element in  
$\mathfrak{l}_{\trBV}$.
By Lemma~\ref{lem:CVt=1}, one can deform $\alpha+\phi+\rho$ in the direction of $\lambda$ up to time $t=1$; 
let us denote the Maurer--Cartan element obtained at time $t=1$ by $e^\lambda. (\alpha  + \phi+\rho)$.

\begin{proposition}\label{prop:MCHypercom}
The Maurer--Cartan element $e^\lambda. (\alpha + \phi+\rho)$  vanishes on $\bar{T}^c(\delta)\oplus s\, \bar{T}^c(\delta)$. Its component on $H^{\bullet+1}(\mathcal{M}_{0,n+1})$ provides us with a  homotopy hypercommutative algebra. 
\end{proposition}

\begin{proof}
Using the notations of the previous proof, we just need to prove that $\rho(1)=0$ and that $\phi(1)=0$. 
The solution $\rho(1)=s^{-1}\left(e^\lambda\star(1+s\, \rho)-1\right)$ gives, with $\lambda=-\ln(1+s\, \rho)$, the following result 
$$\rho(1)=s^{-1}\left(
(1+s\, \rho)^{-1}\star (1+s\, \rho)-1
\right)=0\ . $$
Let us denote by $\partial$ the element in $\a_\Delta$ which sends $1$ to $\partial_A$ and the rest to $0$. Then, in the convolution associative algebra $\a_\Delta$, one has 
$$\phi(1)+\partial= 
e^{-\lambda} \star (\phi +\partial) \star e^\lambda = (1+s\, \rho)\star (\phi + \partial) \star (1+s\, \rho)^{-1}
\ , $$
see \cite[Section~$2$]{DotsenkoShadrinVallette14} for more details.
Since $r$ is a homotopy trivialisation of $\phi$,  the right-hand side is equal to $\partial$, see \cite[Section~$6$]{DotsenkoShadrinVallette14}, which proves the first statement. 

The second claim follows from the explicit shape of the differential of the operad $\trBV_\infty$. 
\end{proof}

Let us now interpret this result in a functorial operadic way. All the previous arguments and computations hold true, in the same way, in the convolution $L_\infty$-algebra 
$$ \big(
\Hom_\Sy(H^{\bullet+1}(\mathcal{M}_{0,n+1})\oplus \bar{T}^c(\delta) \oplus s\, \bar{T}^c(\delta), \trBV_\infty),  d_*+(d_1)^*, \ell_2,  \ell_3, \ldots
\big)\ .$$
The identity endomorphism of the operad $\trBV_\infty$ corresponds to a Maurer--Cartan element that we denote by $\alpha+\phi+\rho$, as above. With $\lambda=-\ln (1+s\, \rho)$, the new Maurer--Cartan $e^\lambda. (\alpha + \phi+\rho)$ obtained under the gauge action corresponds to a new endomorphism of the dg operad $\trBV_\infty$. Proposition~\ref{prop:MCHypercom} shows that it  vanishes on the components $\bar{T}^c(\delta) \oplus s\, \bar{T}^c(\delta)$. So it actually defines a morphism of dg operads 
$$G\ : \ \HyperCom_\infty \to \trBV_\infty\ ,$$ 
which we call the \emph{Givental morphism}. 

\begin{proposition}
 The Givental morphism is a quasi-isomorphism of dg operads.
\end{proposition}

\begin{proof}
Integrating 
the formula 
\begin{eqnarray*}
\alpha'(t)
=
\frac12  \ell_3\big(\alpha(t)+\phi(t), \alpha(t)+\phi(t), \lambda\big) + 
\frac16  \ell_4\big(\alpha(t)+\phi(t), \alpha(t)+\phi(t), \alpha(t)+\phi(t), \lambda\big) + \cdots \ ,
\end{eqnarray*}
from the proof of Lemma~\ref{lem:CVt=1} shows that the restriction of the Givental morphism 
on the space of generators has the following shape 
$$
H^{\bullet}(\mathcal{M}_{0,n+1}) \xrightarrow{\id + \zeta} H^{\bullet}(\mathcal{M}_{0,n+1}) \oplus 
\TTT\big( H^{\bullet}(\mathcal{M}_{0,n+1})\oplus s^{-1}\bar{T}^c(\delta) \oplus \underbrace{\bar{T}^c(\delta)}_{\geqslant 1}\big) \ , $$
where the image of  the map $\zeta$ lives in the linear span of trees with at least one vertex labelled by a $\delta^k$. 
Therefore the composite of the Givental morphism $G$ with the projection onto $\HyperCom_\infty$ is equal to the identity:
$$\id : \HyperCom_\infty \xrightarrow{G} \trBV_\infty \epi  \HyperCom_\infty\ . $$

Let us now prove that the homology of the dg operad $\trBV_\infty$ is isomorphic to the operad $\HyperCom$. 
We consider the following filtration of the dg operad $\trBV_\infty$: 
$$F_n:= \TTT\big( H^{\bullet}(\mathcal{M}_{0,n+1})\oplus s^{-1}\bar{T}^c(\delta) \oplus \bar{T}^c(\delta)\big)^{(\geqslant -n)}$$
made up of linear combination of trees with at least $-n$ vertices. Since the differential $d_1+\allowbreak d_2'+\allowbreak d_2+\allowbreak d_3+\cdots$ amounts to splitting the generators, it preserves this filtration.
This filtration is exhaustive and bounded below at fixed arity $n\ge 1$: at fixed homological degree $k\ge 0$, there exists an $N$ for which $F_N(\trBV_\infty(n)_k)=0$. So the associated spectral sequence converges to the homology of the dg operad $\trBV_\infty$. The differential of the first page of the spectral sequence is equal to $d_1$. The associated homology is isomorphic to $\TTT\big( H^{\bullet}(\mathcal{M}_{0,n+1})\big)$. So the second page of the spectral sequence is isomorphic to $\Omega \HyperCom^{\ac}$,which is the Koszul resolution of $\HyperCom$. Finally, the spectral sequence collapses at the third page, which is isomorphic to $E^2_{-1, \bullet}\cong\HyperCom$. 

On the homology level, the Givental morphism sits inside the following retract 
$$\id : \HyperCom \xrightarrow{H_\bullet(G)} \HyperCom \epi  \HyperCom\ .$$
So the  morphism $H_\bullet(G)$ is a monomorphism. At fixed arity $n\ge 1$ and homological degree $k\ge 0$, the component $\HyperCom(n)_k$ of the operad 
$\HyperCom$ is finite dimensional, which concludes that the morphism $H_\bullet(G)$ is an isomorphism.
\end{proof}

This proposition lifts the main result of \cite{KhoroshkinMarkarianShadrin13} on the cofibrant resolution level. 

\subsection{Functorial  Givental action}
The shape of the differential of the minimal model $\sBV_\infty$ of the operad $\BV$ shows that  the projection of its space of  generators onto its first summand provides us with another morphism of dg operads: 
$$\sBV_\infty \to \HyperCom_\infty \ , $$
which induces a morphism of dg operads 
$$\trBV_\infty
 \to \HyperCom_\infty\vee\, T(\bar{T}^c(\delta)) \ . $$
Algebras over the coproduct operad $\HyperCom_\infty\vee\, T(\bar{T}^c(\delta))$  are made up of a homotopy hypercommutative algebra $\alpha$ and a trivialisation $\rho$ of the trivial circle action.
The above morphism of dg operads amounts just to considering such a datum as a homotopy BV-algebra with a trivial higher circle action and a homotopy trivialisation of it.

We consider the composite of these two dg operad maps
$$ \xymatrix{
\widetilde{G} : \HyperCom_\infty  \ar[r]  
 & \trBV_\infty  \ar[r]   &    \HyperCom_\infty\vee\, T(\bar{T}^c(\delta)) }\ ,
$$
and we define the endomorphism $\mathcal G$ of the dg operad $\HyperCom_\infty\vee\, T(\bar{T}^c(\delta))$ by the following coproduct diagram 
$$\xymatrix{  \HyperCom_\infty \ar[r]^(0.42){\widetilde{G}} \ar[d]& \HyperCom_\infty\vee\, T(\bar{T}^c(\delta)) \\
\HyperCom_\infty\vee\, T(\bar{T}^c(\delta)) \ar@{..>}[ur]^(0.47){\mathcal G}   &  \ar[u]  \ar[l] T(\bar{T}^c(\delta)) \ . } $$

\begin{proposition}
The endomorphism $\mathcal G$ is  a non-trivial automorphism of dg operad: 
$$\mathcal G \in \mathrm{Aut}\big( \HyperCom_\infty\vee\, T(\bar{T}^c(\delta))\big)\ . $$
\end{proposition}

\begin{proof}
The arguments of the previous proof show that the Givental morphism is a monomorphism, so is the morphism $\widetilde{G}$ and then the morphism $\mathcal G$. At fixed arity $n\ge 1$ and homological degree $k\ge 0$, the component of the operad 
$\HyperCom_\infty\vee\, T(\bar{T}^c(\delta))$ is finite dimensional, which concludes the proof.
\end{proof}

 Recall that the data of a homotopy trivialisation $\rho$ of the trivial circle action  is equivalent to a series $R(z)\in 1+z\End(A)[[z]]$ commuting with the differential $d_A$. Such elements are precisely exponentials of elements $r(z)$ that can be used to define the infinitesimal Givental action on homotopy hypercommutative algebras, see Definition~\ref{def:DefinitionForHomotopyHypercom}. Pulling back such a data $(\alpha, \rho)$ by the above operad map produces a new homotopy hypercommutative algebra. When starting from a strict hypercommutative algebra, the result is nothing but the hypercommutative algebra obtained via the Givental group action. 

\begin{theorem}\label{thm:MainII}
Let $\alpha$ be a hypercommutative algebra and let $R(z)$ be a trivialisation of the trivial circle action. The pullback hypercommutative algebra ${\widetilde{G}}^*(\alpha, R(z)-1)$  is equal to the hypercommutative algebra, or CohFT, obtained by the Givental group action of the element $R(z)$. 
\end{theorem}

\begin{proof}
In Theorem~\ref{thm:MainGiv=LinftyAction}, we proved the equality between the infinitesimal Givental action of $r(z)$ on the CohFT $\alpha$ and the infinitesimal gauge symmetry $\ell_1^\alpha(r(z))$. The integration at time $t=1$ of
the second one is equal to $\widetilde{G}^*(\alpha, \rho)$ and the integration at time $t=1$ of the first one is equal to the Givental 
group action of $R(z)$ on the CohFT $\alpha$.
\end{proof}

In the same way as at the end of Section~\ref{sec:GaugeInterpretation}, a direct consequence of this theorem is that the Givental group action on hypercommutative algebras extends naturally to  homotopy hypercommutative algebras as follows. The subgroup of the Givental group $1+z\End(A)[[z]]$ (formal Taylor loops of $GL(A)$) made up of series commuting with the differential $d_A$ is the group of $\infty$-isotopies of the trivial circle action.

\begin{definition}\label{def:DefinitionGroupForHomotopyHypercom}
Let $\alpha$ be a homotopy hypercommutative algebra structure on a chain complex $A$  and let $R(z)\in 1+z\End(A)[[z]] $  be a degree $0$ element  commuting with $d_A$. The \emph{higher Givental action} of $R(z)$ on $\alpha$ is defined by:
$$R(z).\alpha:=\widetilde{G}^*(\alpha, R(z)-1)\ . $$
\end{definition}

Notice that, even if $\infty$-isotopies of the trivial circle action form a group, this action is \emph{not} a group action but rather an $\infty$-groupoid action since it is defined by integrating an $L_\infty$-algebra and not a Lie algebra, cf. \cite{Getzler09}. We also remark that Definition \ref{def:DefinitionGroupForHomotopyHypercom} can be equivalently given using the automorphism $\mathcal{G}$, since $\mathcal{G}^*(\alpha, R(z)-1)=(R(z).\alpha, R(z)-1)$. \\

Theorem~\ref{thm:MainII} actually proves an extended version of the claim made by M. Kontsevich in 2003 that the Givental group action on CohFT's is equal to a change of trivialisation of the circle action. 
The results of this section show that 
 the (higher) Givental action on (homotopy) hypercommutative algebras is an ($\infty$-groupoid) action of trivialisations of the trivial circle action.

On the level of  hypercommutative algebras, a proof of Kontsevich's claim is given in \cite{KhoroshkinMarkarianShadrin13}. It is based on lifts of various structures to a bigger underlying space where the circle action is trivialised. 
The final result does not depend on the choices of lifts, but the proof is not functorial. The present work provides us with a functorial proof of Kontsevich's claim. 

To conclude, at the higher homotopy level, the Givental action admits an interpretation which is simpler (as the underlying circle action is trivial, not just trivialised) and fully functorial.

\bibliographystyle{amsalpha}
\bibliography{bib}

\end{document}